\documentclass[10pt]{article}
\usepackage[square,authoryear]{natbib}
\usepackage{graphicx}
\usepackage{marsden_article}

\usepackage{epstopdf}

\usepackage{amssymb}

\usepackage{eufrak}
\usepackage[all]{xy} 
\usepackage{amssymb} 
\usepackage{epsfig}
\usepackage{multicol}
\usepackage{xypic}

\usepackage{amsthm}
\usepackage{amsmath}
\usepackage{wasysym}

\newtheorem*{remark}{Remark}

\usepackage{hyperref}
\hypersetup{
   colorlinks,%
   citecolor=green,%
   linkcolor=blue,%
   urlcolor=red
}

\usepackage{amssymb}
\usepackage{eucal}

\usepackage{eufrak}
\usepackage{graphics}
\usepackage{graphicx}
\usepackage{epsfig}
\usepackage{multicol}
\usepackage{xypic}
\usepackage{amsmath}
\usepackage{wasysym}
\usepackage{dsfont}

\usepackage{eufrak}
\usepackage{graphics}
\usepackage{graphicx}
\usepackage{epsfig}
\usepackage{multicol}
\usepackage{xypic}
\usepackage{amsmath}

\newcommand{\lp}{\left(}
\newcommand{\rp}{\right)}
\newcommand{\lc}{\left\{}
\newcommand{\rc}{\right\}}
\newcommand{\der}{\partial}

\newcommand{\bra}{\langle}
\newcommand{\ket}{\rangle}
\newcommand{\R}{\mathds{R}}      
\newcommand{\F}{\mathds{F}}

\newcommand{\Flder}{\rightarrow}

\usepackage{amssymb}

\newcommand{\proa}{A^{\ast}G \mbox{$\;$}_{\tau^{\ast}} \kern-3pt\times_\alpha
G \mbox{$\;$}_\beta \kern-3pt\times_{\tau^{\ast}} A^{\ast}G}

\newcommand{\we}{\wedge}

\newcommand{\DQ}{\Delta_Q}
\newcommand{\LQ}{L_{\Delta_Q}}
\newcommand{\ext}{\mathfrak{L}}

\newcommand{\de}{\mathbf{d}}
\newcommand{\Dirac}{\mathbf{d}_D}

\begin{document}
\title{Dirac Structures in Vakonomic Mechanics}
\vspace{-0.2in}

\author{Fernando Jim\'enez\thanks {Email: fjimenez@ma.tum.es}
 \\[2mm] Zentrum Mathematik\\
TU M\"unchen
\\ Boltzmannstr. 3, 85747 Garching, Germany
\vspace{2mm}
 \and 
 \vspace{2mm}
Hiroaki Yoshimura\thanks{Email: yoshimura@waseda.jp}
\\ Department of Applied Mechanics and Aerospace Engineering\\ \&\\ 
Institute of Nonlinear Partial Differential Equations
\\ Waseda University
\\ Okubo, Shinjuku, Tokyo, 169-8555, Japan\\
}

\maketitle
\vspace{-0.3in}
\begin{center}
\abstract{\vspace{0.3cm}
In this paper, we explore dynamics of the nonholonomic system called {\it vakonomic mechanics} in the context of Lagrange-Dirac dynamical systems using a Dirac structure and its associated  Hamilton-Pontryagin variational principle. We first show the link between vakonomic mechanics and nonholonomic mechanics from the viewpoints of Dirac structures as well as Lagrangian submanifolds. Namely, we clarify that Lagrangian submanifold theory cannot represent nonholonomic mechanics properly, but vakonomic mechanics instead. Second, in order to represent vakonomic mechanics, we employ the space $TQ\times V^{\ast}$, where a {\it vakonomic Lagrangian} is defined from a given Lagrangian (possibly degenerate) subject to nonholonomic constraints. Then, we show how {\it implicit vakonomic Euler-Lagrange equations} can be formulated by the {\it Hamilton-Pontryagin variational principle} for the vakonomic Lagrangian on the extended Pontryagin bundle $(TQ\oplus T^{\ast}Q)\times V^{\ast}$.  Associated with this variational principle, we establish a  Dirac structure on $(TQ\oplus T^{\ast}Q)\times V^{\ast}$ to define an {\it intrinsic vakonomic Lagrange-Dirac system}. Furthermore, we establish another construction for the vakonomic Lagrange-Dirac system using a Dirac structure on $T^{\ast}Q\times V^{\ast}$, where we introduce a {\it vakonomic Dirac differential}. Lastly, we illustrate our theory of vakonomic Lagrange-Dirac systems by some examples such as the vakonomic skate and the vertical rolling coin.
\vspace{2mm}
\noindent
}
\end{center}

\\

\tableofcontents

\section{Introduction}
\paragraph{Some Backgrounds.} 
In conjunction with optimal control design, much effort has been concentrated upon exploring geometric structures and variational principles of constrained systems (see, for instance, \cite{Lan1949, Arnold, GiaHil, Jurdjevic, MarRat1999, Bloch2003}).  
The motion of such constrained systems may be subject to a nontrivial distribution on a configuration manifold. For the case in which the given distribution is integrable in the sense of Frobenius's theorem, the constraint is called holonomic, otherwise nonholonomic. It is well known that equations of motion for Lagrangian systems with holonomic constraints can be formulated by Hamilton's variational principle by incorporating holonomic constraints into an original Lagrangian through Lagrange multipliers. On the other hand, Hamilton's variational principle does not yield correct equations of motion for mechanical systems with nonholonomic constraints, but induces different mechanics instead. The correct equations of motion for nonholonomic mechanics can be developed from the Lagrange-d'Alembert principle. In other words, there are two different mechanics associated with systems with nonholonomic constraints. The first one is based on the Lagrange-d'Alembert principle and the corresponding equations of motion are called {\it nonholonomic mechanics}. The second one is called {\it vakonomic mechanics} ({\it mechanics of \textbf{v}ariational \textbf{a}xiomatic \textbf{k}ind}), which is purely variational and was developed by \cite{Kozlov}; the name of vakonomic mechanics was coined by \cite{Arnold}. Needless to say, both approaches are essentially different from the other: interesting comparisons between both of them can be found in \cite{Lewis,CLMM2003}.

Nonholonomic mechanics has been studied from the viewpoints of Hamiltonian, Lagrangian as well as Poisson dynamics (see \cite{KM1997}). Indeed, nonholonomic mechanics has
many applications to engineering, robotics, control of satellites, etc., since it seems to be appropriate to
model the dynamical behavior of phenomena such as rolling rigid-body, etc. (see \cite{NeFu1972}). On the other hand, vakonomic mechanics appears in some problems of optimal control theory (related to sub-Riemannian geometry) (\cite{BlochCrouch,Brockett}), economic growth theory (\cite{LeMa1998}), motion of microorganisms
at low Reynolds number (\cite{Jair}), etc. A geometric unified approach was developed in \cite{LeMaMa2000}.
\medskip

In mechanics, one usually starts with a configuration manifold $Q$; Lagrangian mechanics deals with the tangent bundle $TQ $, while Hamiltonian mechanics with the cotangent bundle $T^{\ast}Q$. It is known that nonholonomic and vakonomic mechanics can be described on extended spaces because of the presence of Lagrange multipliers.
An interesting geometric approach to Lagrangian vakonomic mechanics on $TQ\times\R^m$ may be found in \cite{BendeDie2005}, while an approach on $T(Q\times\R^m)$ may be found in \cite{MarCordeLe2000}. In particular, since an extended Lagrangian on $TQ\times\R^m$ or $T(Q\times\R^m)$ is clearly {\it degenerate}, we have to explore its dynamics by using Dirac's theory of constraints (see \cite{Dirac1950}). Another interesting approach may be found in \cite{CLMM2003}, where the authors depart from $TQ\oplus T^{\ast}Q$, and its submanifold $W_0=\DQ\times_Q T^{\ast}Q$, where $\DQ\subset TQ$, in order to develop an intrinsic description of vakonomic mechanics. 
\medskip

As shown in \cite{YoMa2006a}, degenerate Lagrangian systems with nonholonomic constraints may be described, in general, by a set of implicit differential-algebraic equations, where a key point in the formulation of such implicit systems is to  make use of the {\it Pontryagin  bundle} $TQ \oplus T^{\ast}Q$, namely the fiber product (or Whitney) bundle $TQ \oplus T^{\ast}Q$. To the best of our knowledge, the Pontryagin  bundle was first investigated in \cite{SkRu1983} to aid in the study of the degenerate Lagrangian systems, which is the case that we also treat in the present paper. The iterated tangent and cotangent spaces $TT^{\ast}Q$, $T^{\ast}TQ$,  and $T^{\ast}T^{\ast}Q$ and the relationships among these spaces were investigated by \cite{Tu1977} in conjunction with the generalized Legendre transform, where a symplectic diffeomorphism $\kappa_Q: TT^{\ast}Q \to T^{\ast}TQ$ plays an essential role in understanding Lagrangian systems in the context of Lagrangian submanifolds.  The relation between these iterated spaces and the Pontryagin bundle was also discussed in \cite{CHHM1998}. Furthermore, \cite{Cour1990b} investigated the iterated spaces $TT^{\ast}Q$, $T^{\ast}TQ$,  and $T^{\ast}T^{\ast}Q$ in conjunction with the tangent Dirac structures. 
\medskip

The notion of Dirac structures was developed by \cite{CoWe1988,Dorfman1987} as a unified structure of pre-symplectic and Poisson structures, where the original aims of these authors were to formulate the dynamics of constrained systems, including constraints induced from degenerate Lagrangians, as in \cite{Dirac1950, Dirac1964}, where we recall that Dirac was concerned with degenerate La\-gran\-gians, so that  the image $P \subset T^{\ast}Q$ of the Legendre transformation, called the set of {\it primary constraints} in the language of Dirac, need not be the whole space. 
 The canonical Dirac structures can be given by the graph of the bundle map associated with the canonical symplectic structure or the graph of the bundle map  associated with the canonical Poisson structure on the cotangent bundle, and hence it naturally provides a geometric setting for Hamiltonian mechanics. It was already shown by \cite{Cour1990a} that Hamiltonian systems can be formulated in the context of Dirac structures, however, its application to electric circuits and mechanical systems with nonholonomic constraints was studied in detail by \cite{VaMa1995}, where they called the associated Hamiltonian systems with Dirac structures {\it implicit Hamiltonian systems}. On the other hand, \cite{YoMa2006a} explored on the Lagrangian side to clarify the link between an induced Dirac structure on $T^{\ast}Q$ and a degenerate Lagrangian system with nonholonomic constraints and they developed a notion of {\it implicit Lagrangian systems} as a Lagrangian analogue of implicit Hamiltonian systems. Moreover, the associated variational structure with implicit Lagrangian systems was investigated in \cite{YoMa2006b}, where it was shown that the Hamilton-Pontryagin principle  provides the standard implicit Lagrangian system. Another recent development that may be relevant with the Dirac theory of constraints was explored by \cite{CeEtFe2011} by emphasizing the duality between the Poisson-algebraic and the geometric points of view, related to Dirac's and of Gotay and Nester's work.

\paragraph{Goals of the Paper.}
The main purpose of this paper is to explore vakonomic mechanics, in the Lagrangian setting, both in the context of the Dirac structure and its associated variational principle called the Hamilton-Pontryagin principle. Another important point that we will clarify is the link between Dirac structures and Lagrangian submanifolds for the case of vakonomic mechanics. The organization of the paper is given as follows:

In \S2, we will briefly introduce the geometric setting of the iterated tangent and cotangent bundles as well as the Pontryagin bundle. In \S3, we will shortly review the Lagrangian submanifold theory for mechanics and will show that nonholonomic mechanics cannot be formulated on Lagrangian submanifolds, since the pullback of a symplectic two-form to the submanifold does not vanish. In \S4 we will review Dirac structures in nonholonomic mechanics, by using the induced Dirac structure on the cotangent bundle and we will show how a degenerate Lagrangian system can be developed in the context of Dirac structures, together with the associated Lagrange-d'Alembert principle. In \S5, we will consider the extended tangent bundle $TQ\times V^{\ast}$, where an extended Lagrangian $\mathfrak{L}$, called {\it vakonomic Lagrangian}, is defined in association with a given Lagrangian $L$ on $TQ$ and with nonholonomic constraints. Then we will show that the vakonomic dynamics on $(TQ\oplus T^{\ast}Q)\times V^{\ast}$ can be obtained by the Hamilton-Pontryagin principle for $\mathfrak{L}$, which yields the {\it implicit vakonomic Euler-Lagrange equations}. In parallel with this variational setting, taking advantage of the presymplectic structures constructed on $(TQ\oplus T^{\ast}Q)\times V^{\ast}$, we will illustrate how the vakonomic analogue of the Lagrange-Dirac systems  can be intrinsically developed by making use of the Dirac structure on $(TQ\oplus T^{\ast}Q)\times V^{\ast}$. We shall also show another construction of the vakonomic Lagrange-Dirac system by employing a Dirac structure on $T^{\ast}Q\times V^{\ast}$. To do this, we make use of the Dirac differential of $\mathfrak{L}$, where we introduce two maps $\widehat{\Omega}^{\flat}: T(T^{\ast}Q \times V^{\ast}) \to T^{\ast}(T^{\ast}Q \times V^{\ast})$ and $\tilde{\gamma}_Q: T^{\ast}(T^{\ast}Q \times V^{\ast}) \to T^{\ast}(TQ \times V^{\ast})$ among the iterated bundles $T^{\ast}(T^{\ast}Q\times V^{\ast})$,  $T(T^{\ast}Q\times V^{\ast})$ and $T^{\ast}(TQ\times V^{\ast})$. It will be proved that the vakonomic Lagrange-Dirac system leads to the implicit vakonomic Euler-Lagrange equations. The section is closed with the main result of this paper, Theorem \ref{TEO}, where we summarize vakonomic mechanics can be formulated by Dirac structures as well as the Hamilton-Pontryagin variational principle. In \S6, we will demonstrate our theory by some examples such as the vakonomic particle, the vakonomic skate and the vertical rolling disk on a plane. In \S7, we will give some concluding remarks and future works.
\smallskip

\paragraph{Hamilton's Principle for Holonomic Lagrangian Systems.}
Before going into details, let us briefly recall the variational principle for constrained Lagrangian systems.
First consider the case in which holonomic constraints are given. 
Let $Q$ be a smooth $n-$dimensional manifold. Let $L:TQ\Flder\R$ be a Lagrangian and let $\Delta_{Q}$ be a constraint distribution on $Q$ given for each $q \in Q$ as  
\begin{equation*}\label{distribution}
\Delta_Q(q)=\lc v_q\in TQ\,|\,\bra \mu^{\alpha}(q),v_q\ket=0,\,\alpha=1,...,m <n\rc,
\end{equation*}
where $\mu^{\alpha}$ are $m$ independent one-forms that form the basis for the annihilator $\Delta_Q^{\circ}\subset T^{\ast}Q$. In this paper, we assume that every distribution is {\it regular}, namely, it has constant rank at each point and is smooth unless otherwise stated.
It follows from Frobenius's theorem that  $\Delta_{Q}$ is {\it integrable} or {\it holonomic} if for any vector fields $X, Y$ on $Q$ with values in $\Delta_{Q}$, $[X,Y]$ is a vector field that takes values in $\Delta_{Q}$. Then, the submanifold $\Delta_{Q}$ may be described by a foliation $N \subset Q$ such that, for each $q \in Q$,
\[
\Delta_{Q}(q)=T_{q}N,
\]
where there exist smooth {\it local} functions $\varphi^{\alpha}: Q \to \mathbb{R}$ as
\[
\varphi^{\alpha}(q)=\mathrm{const},\quad \alpha=1,...,m;
\]
and $\mu^{\alpha}=\de\varphi^{\alpha}$ at each point in $q \in Q$. 
\medskip

Let us define an extended Lagrangian $\mathcal{L}: TQ \times \mathbb{R}^{m} \to \mathbb{R}$ by
\[
\mathcal{L}(q,\dot{q},\lambda):=L(q,\dot{q})+\sum_{\alpha=1}^{m}\lambda_{\alpha} \varphi^{\alpha}(q),
\]
where $(q,\dot q)$ are the local coordinates of $TQ$ and $\lambda_{\alpha}$ are the Lagrange multipliers, which may be regarded as new variables. It follows from Hamilton's principle that the stationarity condition for the action integral
\[
\int_{t_1}^{t_2}\mathcal{L}(q(t),\dot{q}(t),\lambda(t)) dt
\] 
induces equations of motion for the holonomic Lagrangian mechanical systems (see \cite{GiaHil, Yo2008}):
\begin{align*}
&\frac{d}{dt}\frac{\der L}{\der\dot q}=\frac{\der L}{\der q}+\sum_{\alpha=1}^{m} \lambda_{\alpha}\frac{\der\varphi^{\alpha}(q)}{\der q},\\[3mm]
&\varphi^{\alpha}(q)=0.
\end{align*}
Regarding the repeated indices such as $\alpha$ in the above equations, we will employ Einstein's summation convention in this paper unless otherwise noted.

\paragraph{Conventional Setting for Vakonomic Systems.}\label{usual}
Let $\phi^{\alpha}:TQ\Flder \R$ be a set of $m$ smooth functions, where $\alpha=1,...,m$, by which the constraints $\phi^{\alpha}=\mathrm{const}$  define a $(2n-m)-$dimensional submanifold $\Delta_Q\subset TQ$. 

As in holonomic Lagrangian systems, let us define an extended Lagrangian $\ext:TQ\times\R^m\Flder\R$ by
\begin{equation}\label{Ext}
\ext(q,\dot{q},\lambda):=L(q,\dot{q})+\lambda_{\alpha}\phi^{\alpha}(q,\dot{q}).
\end{equation}

Now, it is known that Hamilton's principle for the action functional in the above does {\it not} provide correct equations of motion for the nonholonomic Lagrangian mechanical system, but some other dynamics called {\it vakonomic mechanics} (see \cite{Arnold}). The correct equations of motion for nonholonomic Lagrangian mechanics can be  given by {\it Lagrange-d'Alembert principle}.
\medskip

Again, let $(q,\dot q,\lambda)$ be local coordinates for $TQ\times\R^m$ and consider the action functional given by
\[
\int_{t_{1}}^{t_{2}}\ext(q(t),\dot q(t),\lambda(t))\,dt.
\]
Keeping fixed the endpoints of the curve $q(t), \; t\in I=[t_{1},t_{2}] $, i.e., $q(t_{1})$ and $q(t_{2})$ fixed, whereas $\lambda(t_{1})$ and $\lambda(t_{2})$ of $\lambda(t)$ are allowed to be free, the stationary condition of the above action functional provides the {\bfi vakonomic Euler-Lagrange equations}
\begin{equation*}
\begin{split}
\frac{d}{dt}\lp\frac{\der\ext}{\der\dot q}\rp-\frac{\der\ext}{\der q}&=0,\\
\frac{\der\ext}{\der\lambda_{\alpha}}&=0,
\end{split}
\end{equation*}
which induce the usual equations of motion of vakonomic dynamics:
\begin{equation}\label{VDynamics}
\begin{split}
\frac{d}{dt}\lp\frac{\der L}{\der\dot q}+\lambda_{\alpha}\frac{\der\phi^{\alpha}}{\der\dot q}\rp&=\frac{\der L}{\der q}+\lambda_{\alpha}\frac{\der\phi^{\alpha}}{\der q},\\[3mm]
\phi^{\alpha}(q,\dot q)&=0.
\end{split}
\end{equation}
\medskip

Now, let us illustrate the vakonomic setting with several applications.

\paragraph{Optimal Control Theory.}\label{octh}

An optimal control problem is described by the following data (see \cite{deLedeMS-MA2007, JimKobdeM2012}): a configuration space $B$ giving the state variables of the system, a fiber bundle $\pi:N\Flder B$ where fibers describe the control variables, a vector field $Y:N\Flder TB$ along the projection $\pi$, and a {\it cost function} $C:N\Flder\R$. We consider the solutions of the optimal control problem the curves $\gamma:I\subset\R\Flder N$ such that $\pi\circ\gamma $ has fixed endpoints (that is, if $b(t)$ is a curve, then $b(t_{1})$ and $b(t_{2})$ have fixed values), extremize the action functional
\[
\int_{t_{1}}^{t_{2}}C(\gamma(t))\,dt,
\]
and satisfy the differential equation
\[
\frac{d}{dt}(\pi\circ\gamma)=Y\circ\gamma,
\]
which rules the evolution of the state variables.

It is easy to show that this is indeed a vakonomic problem on the manifold $N$. The constraint submanifold $M\subset TN$, given by the above-mentioned differential equation, is defined by
\[
M=\lc v_n\in TN\,|\,T\pi(v_n)=Y(n)\,\,,\,n\in N\rc.
\]
The previous  description of an optimal control problem determines the following commutative diagram:
\[
\xymatrix{
N\ar[rr]^{Y}\ar[dr]_{\pi} & &TB\ar[dl]^{\tau_B}\\
  & B&
}
\]
In the above, $\tau_B:TB\Flder B$ is the canonical projection. Notice that $M$ plays the role of $\DQ$, $TN$ the role of $TQ$ and $C$ the role of the Lagrangian function $L$ in the setting of vakonomic mechanics.

\medskip

In conjunction with optimal control, we remark that Pontryagin's maximum principle is the machinery that gives necessary conditions for solutions of optimal control problems (see \cite{PBGM1962, Sussmann1998}), which is relevant with variational principles for vakonomic dynamics in this paper.

\paragraph{Sub-Riemannian Geometry.}
A sub-Riemannian structure on a manifold is a generalization of a Riemannian structure, where
the metric is only defined on a vector subbundle of the tangent bundle. In such a
case, the notion of length is only assigned to a subclass of curves, namely, curves with tangent
vectors belonging to the vector subbundle at each point (see \cite{Lang,Mont}). More precisely, consider an
$n$-dimensional manifold $Q$ equipped with a smooth distribution $\Delta_{Q}(q)$ with constant rank $n-m$ at each point $q \in Q$. A sub-Riemannian
metric on $\DQ$ consists of giving a positive definite quadratic form $g_q$ on $\DQ$ smoothly varying in $Q$. We will say that a piecewise smooth curve $\gamma:I=[t_{1},t_{2}]\subset\R\Flder Q$ is {\it admissible} if $\dot\gamma(t)\in\DQ$ for all $t \in I$. Using the metric $g$, it is possible to define the length $l(\gamma)$
\[
l(\gamma)=\int_{t_{1}}^{t_{2}}\sqrt{g(\dot\gamma(t),\dot\gamma(t))}\,dt,
\]
for admissible curves $\gamma: I\Flder Q$. From this definition, we can define the distance between two points $q_1,q_2\in Q$ as
\[
d(q_1,q_2)=\mbox{inf}_{\gamma}\lp l(\gamma)\rp,
\]  
if there exists admissible curves connecting $q_1$ and $q_2$. A curve which realizes the distance between two points is called a {\it minimizing sub-Riemannian geodesic}. Let $\mu^1,...,\mu^m$ be a basis of one-forms for the annihilator $\DQ^{\circ}$. Then, an admissible path must verify the nonholonomic constraints
\begin{equation}\label{subr}
\bra\mu^{\alpha}(\gamma),\dot\gamma\ket=0,\,\,\,\alpha=1,...,m.
\end{equation}
Therefore, it is clear that the problem to minimize sub-Riemannian geodesics is
exactly the same as the vakonomic problem determined by the Lagrangian $L=\frac{1}{2}\,g$ and with the nonholonomic constraints \eqref{subr}.

\section{Iterated Tangent and Cotangent Bundles} \label{iterated_section}

In this section we recall some basic geometry of the spaces
$TT^{\ast}Q$,  $T^{\ast}T^{\ast}Q$ and $T^{\ast}TQ$, as well as the
Pontryagin bundle $TQ \oplus T^{\ast}Q$. These spaces are needed for the construction of 
Lagrangian mechanics on the tangent bundle $TQ$ and  Hamiltonian 
mechanics on the cotangent bundle $T^{\ast}Q$ in the context of Lagrangian submanifold theory. In particular, there are two diffeomorphisms among $T^{\ast}TQ$, $TT^{\ast}Q$ and 
$T^{\ast}T^{\ast}Q$ that were originally developed by \cite{Tu1977} for the generalized Legendre transform.

\paragraph{Diffeomorphism between $TT^{\ast}Q$ and $T^{\ast}TQ$.}
Now, we are going to define a natural
diffeomorphism
\[
\kappa_Q: TT^{\ast}Q \to T^{\ast}TQ.
\]
In a local trivialization, $Q$
is represented by an open set $U$ in a linear space $V$, so that
$TT^{\ast}Q$ is represented by $(U \times V^{\ast}) \times (V \times
V^{\ast})$, while $T^{\ast}TQ$ is locally given by $(U \times V)
\times (V^{\ast} \times V^{\ast})$. In this local representation, the
map
$\kappa_Q$ will be given by
\[
(q, p, \delta q, \delta p) \mapsto (q, \delta q, \delta p, p),
\]
where $(q,p)$ are local coordinates
of $T^{\ast}Q$ and $(q,p,\delta{q},\delta{p})$ are the corresponding
coordinates of $TT^{\ast}Q$, while $(q, \delta q, \delta p, p)$ are
the local coordinates of $T^{\ast}TQ$ induced by $\kappa_Q$.

Consider the following two maps:
\[
T{\pi_Q}: TT^{\ast}Q \to TQ, \quad
  \pi_{TQ}: T^{\ast}TQ \to TQ,
\]
which are the obvious maps and recall that $\pi_Q: T^{\ast}Q \to Q$ 
is the cotangent projection. The commutative condition used to define $\kappa_Q$ is the 
 following diagram:
\[
\xymatrix{
TT^{\ast}Q\ar[rr]^{\kappa_{Q}}\ar[dr]_{T\pi_{Q}} & & T^{\ast}TQ\ar[dl]^{\pi_{TQ}}\\
  & TQ&
}
\]
Namely,
\begin{equation*} \label{commute1}
\pi_{TQ} \circ \kappa_Q = T \pi_Q.
\end{equation*}

\paragraph{Local Representation.}  In a natural local trivialization, 
these maps are readily checked to be given by
\[
\begin{split}
T\pi_Q(q,p,\delta q,\delta p)&=(q,\delta q),\\
\pi_{TQ}(q,\delta q, \delta p, p)&=(q, \delta q),\\
\tau_{T^{\ast}Q}(q,p,\delta q,\delta p)&=(q,p).
\end{split}
\]

\paragraph{Diffeomorphism between $T^{\ast}T^{\ast}Q$ and
$TT^{\ast}Q$.} Let $\Omega_{T^{\ast}Q}$ be the canonical symplectic form on the
cotangent bundle $T^{\ast}Q$. There exists a natural
diffeomorphism given by 
\[
\Omega^{\flat}_{T^{\ast}Q}:TT^{\ast}Q \to T^{\ast}T^{\ast}Q,
\]
which is the unique map
that also intertwines two sets of maps:
\[
  \tau_{T^{\ast}Q}: TT^{\ast}Q \to T^{\ast}Q, \quad
  \pi_{T^{\ast}Q}: T^{\ast}T^{\ast}Q \to T^{\ast}Q.
\]
Let $\tau_Q:TQ \to Q$ be the tangent projection and the commutative condition for $\Omega_{T^*Q}^{\flat}$ is given by the following diagram:
\[
\xymatrix{
TT^{\ast}Q\ar[rr]^{\Omega^{\flat}_{T^{\ast}Q}}\ar[dr]_{\tau_{T^{\ast}Q}} & & T^{\ast}T^{\ast}Q\ar[dl]^{\pi_{T^{\ast}Q}}\\
  & T^{\ast}Q&
}
\]
Namely,
\begin{equation*}
\pi_{T^{\ast}Q} \circ \Omega^{\flat}_{T^{\ast}Q}=\tau_{T^{\ast}Q}.
\end{equation*}

\paragraph{Local Representations of Maps.}
As before, in a local trivialization, $T^{\ast}T^{\ast}Q$ is represented by $(U \times V^{\ast}) \times 
(V^{\ast} \times V)$, while $TT^{\ast}Q$ is represented by $(U \times 
V^{\ast}) \times (V \times V^{\ast})$. The map $\Omega^{\flat}_{T^{\ast}Q}$ is 
locally represented by
\[
(q,p,\delta{q},\delta{p}) \mapsto (q,p,-\delta{p},\delta{q}).
\]
Thus, the commutative diagram is verified in a local trivialization since one has
\[
\begin{split}
\pi_{T^{\ast}Q}(q,p,-\delta{p},\delta{q})&=(q,p),\\
\tau_{T^{\ast}Q}(q,p,\delta{q},\delta{p})&=(q,p).
\end{split}
\]
\paragraph{The Diffeomorphism between $T^{\ast}TQ$ and $T^{\ast}T^{\ast}Q$.} With the elements previously defined, we can define a diffeomorphism between $T^{\ast}TQ$ and $T^{\ast}T^{\ast}Q$, namely:
\[
\gamma_Q=\Omega^{\flat}_{T^{\ast}Q}\circ(\kappa_Q)^{-1}:T^{\ast}TQ\Flder T^{\ast}T^{\ast}Q.
\]
Using the local representation of $\Omega^{\flat}_{T^*Q}$ and of $\kappa_Q$, the map $\gamma_{Q}$ is locally given by
\[
\gamma_Q:(q,\delta q,\delta p,p)\mapsto (q,p,-\delta p,\delta q).
\]

\paragraph{The Symplectic Form on $TT^{\ast}Q$.}\label{dOmega}
The manifold $TT^{\ast}Q$ is a symplectic manifold with a special
symplectic form $\Omega_{TT^{\ast}Q}$, which can be defined by two
distinct ways as the exterior derivative of two intrinsic but
different one-forms. 

Now, there exist two one-forms $\lambda$ and $\chi$ given by
\[
\lambda=\left( \kappa_Q \right)^{\ast}\, \Theta_{T^{\ast}TQ},\quad
\chi=\left( \Omega^{\flat}_{T^{\ast}Q} \right)^{\ast} \,
\Theta_{T^{\ast}T^{\ast}Q},
\]
where $\Theta_{T^{\ast}TQ}$ denotes the canonical one-form  on
$T^{\ast}TQ$ and $\Theta_{T^{\ast}T^{\ast}Q}$ the
canonical one-form on $T^{\ast}T^{\ast}Q$.

Using the local coordinates $(q,p)$ and
$(q,p,\delta{q},\delta{p})$ of $T^{\ast}Q$ and
$TT^{\ast}Q$, these two one-forms are denoted by
\[
\lambda=\delta{p}\,dq+p\,d\delta{q}, \quad
\chi=-\delta{p}\,dq+\delta{q}\,dp.
\]
Thus, the symplectic form $\Omega_{TT^{\ast}Q}$ on $TT^{\ast}Q$
associated with $\lambda$ and $\chi$ can be defined by
\begin{equation*} 
\begin{split}
\Omega_{TT^{\ast}Q}&=\mathbf{d}\chi =-\mathbf{d}\lambda \\
&=dq \wedge d{\delta{p}}+d\delta{q} \wedge dp.
\end{split}
\end{equation*}
%
%

%
\paragraph{Pontryagin Bundle $TQ \oplus T^{\ast}Q$.}
Consider the bundle $TQ \oplus T^{\ast}Q$ over $Q$,
that is, the Whitney sum of the tangent bundle and the cotangent
bundle over $Q$, whose fiber at $q \in Q $ is the product $T_qQ \times
T^{\ast}_qQ$. The bundle $TQ \oplus T^{\ast}Q$ is called the
{\it Pontryagin bundle}, see \cite{YoMa2006a}. Again, using a model space $V$ for $Q$ and a chart domain, which is an
open set $U \subset V$, then $TQ \times T^{\ast}Q$ is locally denoted by $U \times V \times U \times V^{\ast}$ and $TQ \oplus T^{\ast}Q$ by $U \times V \times V^{\ast}$. In this local
trivialization, the local coordinates of $TQ \oplus T^{\ast}Q$ are written
\[
(q,v,p) \in U \times V \times V^{\ast}.
\]
Then, the following three projections are naturally defined:
\[
\begin{split}
\mathrm{pr}_{TQ}&:TQ \oplus T^{\ast}Q \to TQ; \quad \,\,\,(q,v,p) \mapsto (q, v ),\\
\mathrm{pr}_{T^{\ast}Q}&:TQ \oplus T^{\ast}Q \to T^{\ast}Q; \quad (q,v,p) \mapsto (q, p ),\\
\mathrm{pr}_{Q}&:TQ \oplus T^{\ast}Q \to Q; \quad \,\,\,\,\,\,\,(q,v,p) \mapsto q.
\end{split}
\]
%
%

All the previous developments may be summarized into the following diagram:

\begin{figure}[h]
\begin{center}
\vspace{0.5cm}
\includegraphics[scale=.6]{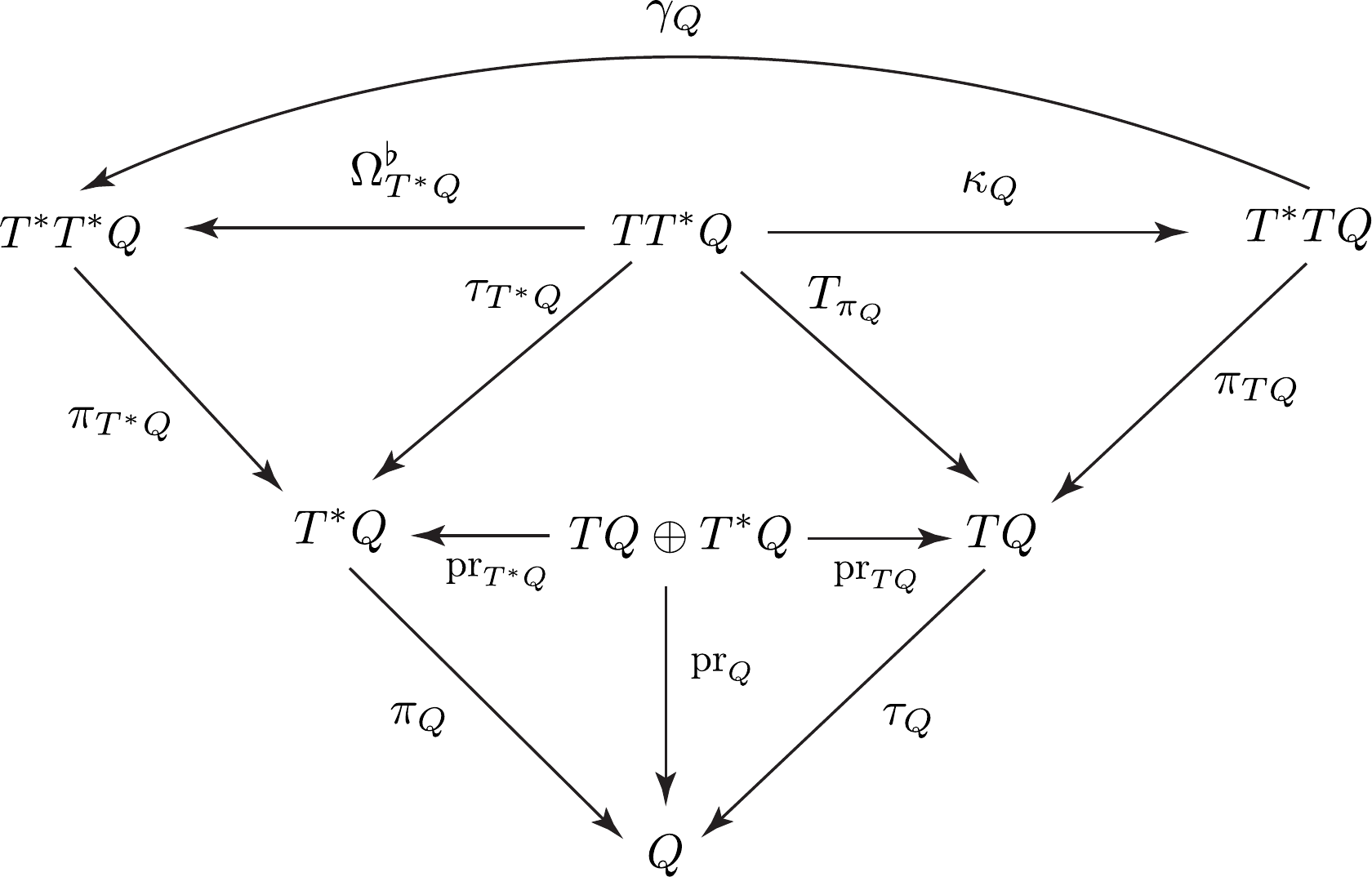}
\caption{Iterated Tangent and Cotangent Bundle Structures}
\label{IteratedTanCoBundle}
\end{center}
\end{figure}
%


\section{Lagrangian Submanifolds in Mechanics}

As introduced in $\S$\ref{iterated_section}, the spaces $TT^{\ast}Q$, $T^{\ast}TQ$, $T^{\ast}T^{\ast}Q$ are interrelated with each other by two symplectomorphisms $\kappa_{Q}:TT^{\ast}Q \to T^{\ast}TQ$ and $\Omega^{\flat}_{T^{\ast}Q}:TT^{\ast}Q \to T^{\ast}T^{\ast}Q$, which play 
essential roles in the construction of the generalized Legendre transformation originally 
developed by \cite{Tu1977}. In this section, we shall see the theory of Lagrangian submanifolds using the geometry of these spaces. For the details, see, for instance, \cite{Wein1971, AbMa1978, Wein1979,LibMarle1987,TuUr1999,YoMa2006b}. 

\paragraph{Lagrangian Submanifolds.} \label{LagSubmani} 
Given a finite-dimensional symplectic manifold $(P,\Omega_P)$, and a submanifold $N$ with canonical inclusion $i_N:N\hookrightarrow P$, then $N$ is a {\it Lagrangian submanifold} if and only if $i_N^{\ast}\Omega_P=0$ and dim$N=\frac{1}{2}$dim$P$. If, locally, $\Omega_{P}=\mathbf{d}\Theta_{P}$, then $i_N^{\ast}\Omega_P=\mathbf{d}i_N^{\ast}\Theta_{P}=0$. So, it follows that there exists a function $f: N \to \mathbb{R}$ (defined locally) such that 
$i_N^{\ast}\Theta_{P}=\mathbf{d}f$. We call $f$ a {\it generating function} of the Lagrangian submanifold $N$. By the Poincar\'e lemma, locally this is always the case.

It is well known that the cotangent bundle $T^{\ast}P$ of a given finite-dimensional smooth manifold $P$, equipped with the symplectic two-form $\Omega_{T^*P}$, is a symplectic manifold $(T^{\ast}P,\Omega_{T^*P})$. Let $\alpha$ be a {\it closed} one-form on $P$. Then, the image of $\alpha$, namely $\Sigma=\mbox{Im}\lp\alpha(P)\rp\subset T^{\ast}P$,  is a Lagrangian submanifold of $(T^{\ast}P,\Omega_{T^*P})$, since $\alpha^{\ast}\Omega_{T^*P}=-\mathbf{d}\alpha=0$. Thus, we obtain a submanifold diffeomorphic to $P$ and transverse to the fibers of $T^{\ast}P$. As to the details, see \cite{AbMa1978}.

A useful extension of the previous construction is given by \cite{Tu1977} in the following theorem.

\begin{theorem}[{\bf Tulczyjew}]\label{Tulczy}
Let $M$ be a smooth manifold, $\tau_M:TM\Flder M$ and $\pi_M:T^{\ast}M\Flder M$ its tangent and cotangent bundle projections respectively. Let $N\subset M$ be a submanifold and $f:N\Flder\R$ a function. Then
\begin{multline*}
    \Sigma_f = \bigl\{ p \in T^{\ast}M \mid \pi _M (p) \in N \text{
        and } \left\langle p, v \right\rangle = \left\langle
        \mathbf{d}f , v \right\rangle \\
      \text{ for all } v \in T N \subset TM \text{ such that } \tau_M (v) =\pi_M (p) \bigr\}
  \end{multline*}
  is a Lagrangian submanifold of $T^{\ast}M$.
\end{theorem}
Here, we shall prove this theorem in a different way from \cite{Tu1977}. Later, we will show the essential difference in geometry  between vakonomic and nonholonomic mechanics by making use of this proof.
\begin{proof}
Assume that $q^i$, $i=1,...,\mbox{dim}\,M$, are local coordinates for $M$. Assume also that $q^a$, $a=1,...,\mbox{dim}\,N$, are adapted local coordinates for $N\subset M$. Using these local coordinates, it is easily shown that $\Sigma_f$ is a submanifold of $T^{\ast}M$ with dimension equal to $\frac{1}{2}$ dim$\,M$. On the other hand, the submanifold $N$ shall be defined by a set of $\alpha=1,...,\mbox{dim}\,M-\mbox{dim}\,N$ constraints in the following way
\begin{equation}\label{tulCons}
\phi^{\alpha}(q)=0.
\end{equation}
To finish the proof, we need to show that $\Omega_{T^*M}$, i.e., the symplectic two-form on $T^{\ast}M$, vanishes when we restrict it to $N$. With this purpose, we take the Lie derivative of \eqref{tulCons}, which provides
\begin{equation}\label{vlift}
(\pounds_{v}\phi^{\alpha})(q):=\frac{d}{dt}\phi^{\alpha}(q(t))=\bra\de\phi^{\alpha}(q(t)),\,v(t)\ket=0,
\end{equation}
where $v(t)=dq(t)/dt$. Since $\bra\, p, v \ket =\bra\de f, v \ket$ for all $p\in T^{\ast}M$, where $\pi_M(p)\in N$, it follows from \eqref{vlift} that
\[
p-\de f=\lambda_{\alpha}\de\phi^{\alpha},
\]      
where $\lambda_{\alpha}$ are Lagrange multipliers. From the previous equation it follows 
\[
p=\de (f+\lambda_{\alpha}\phi^{\alpha}).
\]
Let us introduce the inclusion map $i:\Sigma_f\hookrightarrow T^{\ast}M$. Using Darboux's coordinates $(q,p)$ for $T^{\ast}M$, one has $\Omega_{T^*M}=dq\we dp$. By a direct computation, we arrive to
\[
i^{\ast}\Omega_M=dq\we\de\de(f+\lambda_{\alpha}\phi^{\alpha})= 0,
\]
since $\de^2=0$. This finishes the proof.
\end{proof}
\paragraph{Vakonomic Lagrangian Submanifolds.}
Next, we see how the vakonomic mechanics introduced in $\S$\ref{usual} may be fit into the context of the Lagrangian submanifold theory. 
Particularly, in Tulczyjew's theorem, setting $N=\DQ$, $M=TQ$ and $f=\LQ:\DQ\Flder\R$, we can develop a submanifold of $T^{\ast}TQ$ as
\begin{multline*}
    \Sigma _{\LQ} = \bigl\{ \alpha \in T^{\ast}TQ \mid \pi_{TQ}(\alpha) \in \DQ \text{
        and } \left\langle \alpha, w \right\rangle = \left\langle
        \mathbf{d}\LQ , w \right\rangle \\
      \text{ for all } w \in T \DQ \subset TTQ \text{ such that } \tau_{TQ} (w) =\pi_{TQ} (\alpha) \bigr\}.
  \end{multline*}
Consider the submanifold $\DQ\subset TQ$ defined by the vanishing of the $m$ local constraints $\phi^{\alpha}:TQ\Flder\R$ as
$$
\phi^{\alpha}(q,\dot q)=0,\qquad 1\leq\alpha\leq m < n.
$$
Hence
\[
(T\DQ)^{\circ}=\lc w=(\delta{q},\delta{\dot{q}}) \in T_{(q,\dot q)}TQ \; \middle|\; \left< \de\phi^{\alpha}(q,\dot q), w\right>=0 \rc. 
\]
Note that the above constraints allow us to consider {\it nonlinear nonholonomic constraints}.
Let $i_{\DQ}:\DQ\hookrightarrow TQ$ be the inclusion of the submanifold; one can take an arbitrary extension $L:TQ\Flder\R$ such that $L\circ i_{\DQ}=\LQ$. Applying theorem \ref{Tulczy} we obtain 
\begin{equation}\label{tres.tres}
\begin{split}
\Sigma_{\LQ}&=\left\{
(q, \dot q, \dot{p}, p)\in T^{\ast}TQ\; \biggm|\; \dot{p}=\frac{\partial L }{\partial q} +\lambda_{\alpha}\frac{\partial \phi^{\alpha} }{\partial q},\right. \\
&\qquad \left. \qquad\quad p=\frac{\partial L }{\partial \dot{q}} +\lambda_{\alpha}\frac{\partial \phi^{\alpha} }{\partial \dot{q}},\;\;\phi^{\alpha}(q,\dot q)=0,\,\,\, 1\leq\alpha\leq m \right\}.
\end{split}
\end{equation}
By definition this is a Lagrangian submanifold. Recall the local expression of the diffeomorphism $\kappa_Q: TT^{\ast}Q \to T^{\ast}TQ$; namely,
$$
(q, p, \dot q, \dot{p}) \mapsto (q, \dot q,  \dot{p}, p),
$$
we can construct the Lagrangian submanifold $\kappa^{-1}(\Sigma_{\LQ})\subset TT^{\ast}Q$ as
\begin{equation}\label{vLSub}
\begin{split}
\kappa_Q^{-1}(\Sigma_{\LQ})&=\left\{(q, p, \dot{q}, \dot{p})\in TT^{\ast}Q\;\biggm| \; 
p=\frac{\partial L }{\partial \dot{q}} +\lambda_{\alpha}\frac{\partial \phi^{\alpha} }{\partial \dot{q}},\right. \\
&\qquad \left.\qquad\quad  \dot{p}=\frac{\partial L }{\partial q} +\lambda_{\alpha}\frac{\partial \phi^{\alpha} }{\partial q},\;\;\phi^{\alpha}(q,\dot q)=0,\,\,\, 1\leq\alpha\leq m\right\}\; .
\end{split}
\end{equation}
The solution curve for the dynamics determined by $\kappa_Q^{-1}(\Sigma_{\LQ})\subset TT^{\ast}Q$ is given by $\gamma: I=[t_{1},t_{2}]\subset \R\to T^{\ast}Q$ such that 
$$
\frac{\mathrm{d}\gamma}{\mathrm{d}t}(I)\subset \kappa_Q^{-1}(\Sigma_{\LQ}). 
$$
It is apparent that this verifies the set of differential-algebraic equations in \eqref{VDynamics}.
\medskip

Thus, the Lagrangian submanifold $\kappa_Q^{-1}(\Sigma_{\LQ})$ encloses the vakonomic mechanics, which we shall call the {\bfi vakonomic Lagrangian submanifold}. An analogous discussion can be found in \cite{deLeJimdeM2012}.

This development shows the importance of Theorem \ref{Tulczy} as a key tool in the intrinsic description of Lagrangian and Hamiltonian dynamics, as well as vakonomic mechanics (see \cite{Tu1976a} and \cite{Tu1976b} for further details).

\paragraph{The Lagrange-d'Alembert Principle for Nonholonomic Mechanics.}\label{BriefNHolo}
Let us now formulate the nonholonomic mechanics by using the Lagrange-d'Alembert Principle.
Let $\DQ\subset TQ$ be a regular distribution given by
\begin{equation}\label{distribution}
\Delta_Q(q)=\lc \dot q \in T_{q}Q\;|\;\bra \mu^{\alpha}(q),\dot q\ket=0,\,\alpha=1,...,m <n\rc,
\end{equation}
where the one-forms $\mu^{\alpha}$ are nonintegrable. Note that we consider the case of linear constraints. Given a Lagrangian function $L:TQ\Flder\R$, the dynamics of the nonholonomic mechanical system is determined by the Lagrange-d'Alembert principle, which states that a curve $q(t),\;t \in I$ is an admissible motion of the system if 
\[
\delta \int_{t_1}^{t_2} L(q(t),\dot q(t))\,dt=0,
\]
where we choose variations $\delta q(t)$  that satisfy $\delta q(t)\in\DQ(q(t))$ at each $t$ and with the endpoints of $q(t)$ fixed. Assume that  $\mathrm{rank}\;\DQ(q)=2n-m$ at each point $q$. Since the annihilator $\DQ^{\circ}$ is generated by a set of $m$ independent one-forms as
\[
\DQ^{\circ}(q)=\mbox{span}\lc\mu^{\alpha}=\mu_i^{\alpha}(q)\,dq^i\rc,\,\,\,\alpha=1,...,m,
\]
the equations of motion of the nonholonomic mechanics are locally given by
\begin{equation}\label{NonHEq}
\begin{split}
&\frac{d}{dt}\lp\frac{\der L}{\der\dot q^i}\rp-\frac{\der L}{\der q^i}=\lambda_{\alpha}\mu_i^{\alpha}(q),\\[3mm]
&\mu^{\alpha}_i(q)\dot q^i=0,
\end{split}
\end{equation}
where $\lambda_{\alpha}$ are Lagrange multipliers. 

\paragraph{Nonholonomic Dynamical Submanifolds.}
The dynamics of a mechanical system with nonholonomic constraints \eqref{distribution} may be represented by a submanifold
\begin{multline*}
  \Sigma^{\rm nonh} = \bigl\{ \alpha \in T^{\ast}TQ \mid \pi_{TQ}(\alpha) \in \DQ \text{
        and } \left\langle \alpha, w \right\rangle = \left\langle
        \mathbf{d}L , w \right\rangle \\
      \text{ for all } w \in (T\tau_{Q})^{-1}(\DQ) \subset TTQ \text{ such that } \tau_{TQ} (w) =\pi_{TQ} (\alpha) \bigr\}.
  \end{multline*}
The main difference with the previous case is that the lifted distribution $(T\tau_{Q})^{-1}(\DQ) \subset TTQ$ is given by
\[
(T\tau_{Q})^{-1}(\DQ)=\lc w=(\delta{q},\delta\dot{q}) \in T_{(q,\dot{q})}TQ \; \middle|\; T\tau_{Q}(w)=(q,\delta{q}) \in \Delta_{Q} \rc.
\]
Recall at this point that the input data defining the nonholonomic dynamics is a Lagrangian function $L:TQ\Flder\R$ and a regular distribution $\DQ\subset TQ$, whose annihilator $\DQ^{\circ}$ is spanned by the $m$ independent one-forms $\mu^{\alpha}$. Similar to the definition of the vakonomic Lagrangian submanifold \eqref{vLSub}, we can define the {\bfi nonholonomic dynamical submanifold} as 
\begin{equation}\label{nhdynSub}
\begin{split}
\kappa_Q^{-1}\lp\Sigma^{\rm nonh}\rp&=\left\{(q, p, \dot{q}, \dot{p})\in TT^{\ast}Q\; \biggm|  \;
p=\frac{\partial L }{\partial \dot{q}}, \right. \\
& \quad\quad \left.\dot{p}=\frac{\partial L }{\partial q} +\lambda_{\alpha}\mu^{\alpha}(q),\;\;\mu_i^{\alpha}(q)\,\dot q^i=0,
\,\,\, 1\leq\alpha\leq m \right\}\; .
\end{split}
\end{equation}
The solution curve for the dynamics represented by $\kappa_Q^{-1}\lp\Sigma^{\rm nonh}\rp$ is given by $\sigma: I\subset \R\to Q$ such that 
$$
\frac{\mathrm{d}\sigma}{\mathrm{d}t}(I)\subset\DQ
$$ 
and the induced curve $\gamma: \R\to T^{\ast}Q$ defined by 
$
\gamma=\F L\left(\frac{\mathrm{d}\sigma}{\mathrm{d}t}\right)
$
verifies that 
$$
\frac{\mathrm{d}\gamma}{\mathrm{d}t}(I)\subset \kappa_Q^{-1}\lp\Sigma^{\rm nonh}\rp,
$$
where $\F L:TQ\Flder T^{\ast}Q$ is the Legendre transformation associated with $L$.
 Locally, $\sigma$ must satisfy equations \eqref{NonHEq}. Therefore, $\kappa_Q^{-1}\lp\Sigma^{\rm nonh}\rp$ encloses the nonholonomic dynamics (see also \cite{deLeJimdeM2012}).
\medskip

Now, we have the following proposition for the nonholonomic dynamical submanifold.

\begin{proposition}
The nonholonomic dynamical submanifold $\kappa_Q^{-1}\lp\Sigma^{\rm nonh}\rp$ is not a Lagrangian submanifold of $TT^{\ast}Q$.
\end{proposition}
\begin{proof}
To prove this proposition, recall that the iterated tangent bundle $TT^{\ast}Q$ has a symplectic structure defined by $\Omega_{TT^{\ast}Q}$, which  has the local form $\Omega_{TT^{\ast}Q}=dq\we d\dot p+d\dot q\we dp$, where $(q,p,\dot q,\dot p)$ are local coordinates for $TT^{\ast}Q$.

Using local coordinates, it is clear to see that dim$\,\kappa_Q^{-1}\lp\Sigma^{\rm nonh}\rp=\frac{1}{2}$ dim$\,TT^{\ast}Q$. Let
\[
i:\kappa_Q^{-1}\lp\Sigma^{\rm nonh}\rp\hookrightarrow TT^{\ast}Q
\]
be the inclusion defined in \eqref{nhdynSub} and we have to check if $i^{\ast}\Omega_{TT^{\ast}Q}=0$ {\it does not} hold in order to accomplish the proof. By direct computations, it leads to
\begin{eqnarray*}
i^{\ast}\Omega_{TT^{\ast}Q}&=&\frac{\der ^2L}{\der q^i\der q^j}dq^i\we dq^j+\frac{\der ^2L}{\der q^i\der\dot q^j}dq^i\we d\dot q^j+\lambda_{\alpha}\frac{\der\mu_i^{\alpha}}{\der q^j}d q^i\we  dq^j\\
&+&\frac{\der ^2L}{\der\dot q^i\der q^j}d\dot q^i\we dq^j+\frac{\der ^2L}{\der\dot q^i\der\dot q^j}d\dot q^i\we d\dot q^j.
\end{eqnarray*} 
It follows from symmetric properties that this reduces to
$$
i^{\ast}\Omega_{TT^{\ast}Q}=\lambda_{\alpha}\displaystyle\frac{\der\mu_i^{\alpha}}{\der q^j}dq^i\we dq^j.
$$
Thus we conclude $i^{\ast}\Omega_{TT^{\ast}Q}\neq 0$ since in general ${\der\mu_i^{\alpha}}/{\der q^j} \ne {\der\mu_j^{\alpha}}/{\der q^i}$.
\end{proof}
This last result implies that nonholonomic dynamics cannot be described in terms of Lagrangian submanifolds as we claimed.

\paragraph{Remark.}
Poisson and almost-Poisson manifolds have been widely used in the geometrical description of nonholonomic mechanics (see for instance \cite{CLM1999}, \cite{ILMM1998}, \cite{KM1998}). 
A different notion of {\it Lagrangian submanifold} (based in \cite{LibMarle1987} and \cite{Vaisman1994}) has been developed in \cite{LMV2012} in the context of almost-Poisson geometry in order to construct a universal Hamilton-Jacobi theory including nonholonomic mechanics.

\begin{remark}\rm
As shown in equations \eqref{VDynamics} and \eqref{NonHEq}, dynamical equations of both vakonomic and nonholonomic mechanics are clearly different, which comes from the fact that vakonomic mechanics {\it can be} a Lagrangian submanifold while nonholonomic mechanics {\it cannot}. This difference is clarified from the viewpoint of variational principles; namely, vakonomic mechanics is purely variational since we impose the constraints on the class of curves before applying the variations, while the nonholonomic mechanics is not variational since we impose the constraints to variations of curves after taking variations for the action integral. In other words, for the vakonomic mechanics, the admissible trajectories must lie in $\DQ$ and the admissible variations must be tangent to $\DQ$, while for the nonholonomic mechanics, the admissible variations are generated by infinitesimal variations such that their vertical lift takes values in $T\DQ$. Of course, this reflects the fact that nonholonomic mechanics is the one describing the actual motion of the mechanical systems with nonholonomic constraints, while vakonomic mechanics is not. For this perspective,  see also \cite{Cata} and \cite{Respondek} and references therein. For a historical review on this topic, see \cite{Histo}.
\end{remark}


\section{Dirac Structures in Nonholonomic Mechanics} \label{Dirac_induced_section}

As shown in the previous section, nonholonomic mechanics cannot be represented by Lagrangian submanifolds. In this section, we shall show how nonholonomic mechanics can be described in the context of induced Dirac structures and associated implicit Lagrangian systems, following \cite{YoMa2006a}.

\paragraph{Dirac Structures.} We first recall the definition of a {\it Dirac structure on a vector space} $V$, say finite dimensional for simplicity (see \cite{Cour1990a} and \cite{CoWe1988}). Let $V^{\ast}$ be the dual space of $V$, and $\langle\cdot \, , \cdot\rangle$
be the natural paring between $V^{\ast}$ and $V$. Define the
symmetric paring
$\langle \! \langle\cdot,\cdot \rangle \!  \rangle$
on $V \oplus V^{\ast}$ by
\begin{equation*}
\langle \! \langle\, (v,\alpha),
(\bar{v},\bar{\alpha}) \,\rangle \!  \rangle
=\langle \alpha, \bar{v} \rangle
+\langle \bar{\alpha}, v \rangle,
\end{equation*}
for $(v,\alpha), (\bar{v},\bar{\alpha}) \in V \oplus V^{\ast}$.
A {\it Dirac structure} on $V$ is a subspace $D \subset V \oplus
V^{\ast}$ such that
$D=D^{\perp}$, where $D^{\perp}$ is the orthogonal
of $D$ relative to the pairing
$\langle \! \langle \cdot,\cdot \rangle \!  \rangle$.
\medskip

Now let $M$ be a given manifold and let $TM \oplus T^{\ast}M$ denote the Pontryagin bundle over $M$. A  subbundle $ D \subset TM \oplus T^{\ast}M$ is called a {\it Dirac structure on the bundle} $\tau_M:TM \to M$, when $D(x)$ is a Dirac structure on the vector space $T_{x}M$ at each point $x \in M$. A given two-form $\Omega$ on $M$ together with a distribution $\Delta_{M}$ on $M$ determines a Dirac structure on $M$ as follows\footnote{Precisely speaking, $D$ is called an {\it almost} Dirac structure, while for the case in which the distribution is integrable, $D$ is called a Dirac structure. In this paper, however, we simply call $D$ the Dirac structure unless otherwise stated.}.
\begin{proposition}\label{DProof}
The two-form $\Omega$ determines a Dirac structure $D$ on $M$ whose fiber is given for each $x\in M$ as
\begin{equation}\label{DiracManifold}
\begin{split}
D(x)=\{ (v_{x}, \alpha_{x}) \in T_{x}M \times T^{\ast}_{x}M
  \; \mid \; & v_{x} \in \Delta_{M}(x), \; \mbox{and} \\ 
  & \alpha_{x}(w_{x})=\Omega_{\Delta_{M}}(v_{x},w_{x}) \; \;
\mbox{for all} \; \; w_{x} \in \Delta_{M}(x) \},
\end{split}
\end{equation}
where $\Delta_M\subset TM$ and  $\Omega_{\Delta_{M}}$ is the restriction of $\Omega$ to $\Delta_{M}$.
\end{proposition}

\begin{proof}

The orthogonal of $ D\subset TM\oplus T^*M$ at the point $x\in M$ is given by
\begin{multline*}
    D^{\perp}(x) = \bigl\{ (u_x,\beta_x)\in T_xM\times T_x^*M|\,\,\alpha_x(u_x)+\beta_x(v_x)=0,\,\,\forall\, v_x\in\Delta_M \\
       \mbox{and}\,\,\bra\alpha_x,w_x\ket=\Omega_{\Delta_M}(x)(v_x,w_x)\,\,\,\mbox{for all}\,\,\,w_x\in \Delta_M\bigr\}.
  \end{multline*}
In order to prove that $D\subset D^{\perp}$, let $(v_x,\alpha_x),\, (v_x^{\prime},\alpha_x^{\prime})$ belong to $ \in D(x)$. Then
\[
\bra\alpha_x,\,v_x^{\prime}\ket+\bra\alpha_x^{\prime},\,v_x\ket=\Omega_{\Delta_M}(x)(v_x,v_x^{\prime})+\Omega_{\Delta_M}(x)(v_x^{\prime},v_x)=0,
\]
since $\Omega_{\Delta_M}(x)$ is skew-symmetric. Therefore, $D\subset D^{\perp}$.

To conclude the proof we shall check that $ D^{\perp}\subset D$. Let $(u_x,\beta_x)\in D(x)^{\perp}$. By definition of $D^{\perp}$, we have that
\[
\bra\alpha_x,\,u_x\ket+\bra\beta_x,\,v_x\ket=0
\]
for all $v_x\in \Delta_M$ and $\bra\alpha_x,\,w_x\ket=\Omega_{\Delta_M}(x)(v_x,w_x)$ for all $w_x\in\Delta_M$. Choose $v_x\,,\, u_x\in\Delta_M$ arbitrary vectors.  From $\bra\alpha_x,\,u_x\ket+\bra\beta_x,\,v_x\ket=0$ and the fact that $u_x\in\Delta_M$ is an arbitrary vector we have   that
\[
\Omega_{\Delta_M}(x)(v_x,u_x)+\beta_x(v_x)=0\,\,\,\mbox{for all}\,\,v_x\in \Delta_M,
\]
that is $\beta_x(v_x)=\Omega_{\Delta_M}(x)(u_x,v_x)$ due to the skew-symmetry of $\Omega_{\Delta_M}$. Thus, $(u_x,\beta_x)\in D(x)$ and hence $D^{\perp}\subset D$, as required. Consequently, $D^{\perp}=D$ and the claim holds.
\end{proof}

\paragraph{Remark.}
Of course, the proof above is also valid when $\Delta_M=TM\,\, (\Omega_{\Delta_M}=\Omega)$ and, furthermore, either for pre-symplectic or symplectic two-forms since the key property to accomplish the result is their skew-symmetry. On the other hand, throughout this work we shall define the Dirac structures in a different but equivalent way to proposition \ref{DProof}. Namely, each two-form $\Omega$ on $M$ defines a bundle map $\Omega^{\flat}:TM\Flder T^*M$ by
$\Omega^{\flat} \cdot v =\Omega(v,\cdot)$.
Consequently, we may equivalently define $D(x)$ in \eqref{DiracManifold} as
\begin{equation*}
\begin{split}
D(x)=\{ (v_{x}, \alpha_{x}) \in T_{x}M \times T^{\ast}_{x}M
  \; \mid \;  v_{x} \in \Delta_{M}(x), \; \mbox{and} \;
   \alpha_{x}-\Omega^{\flat}(x) \cdot v_{x} \in \Delta^{\circ}_{M}(x) \;
 \}.
\end{split}
\end{equation*}

We call a Dirac structure $D$  {\it integrable} if  the condition
\[
\langle \pounds_{X_1} \alpha_2, X_3 \rangle
+\langle \pounds_{X_2} \alpha_3, X_1 \rangle+\langle \pounds_{X_3}
\alpha_1, X_2 \rangle=0
\]
is satisfied for all pairs of vector fields and one-forms $(X_1, \alpha_1)$,
$(X_2,\alpha_2)$, $(X_3,\alpha_3)$ that take values in $D$,
where $\pounds_{X}$ denotes the Lie derivative  along the vector
field $X$ on $M$. 

\paragraph{Induced Dirac Structures.} \label{induced}
One of the most important and interesting Dirac structures in mechanics is the one that is induced from kinematic constraints, either holonomic or nonholonomic. This Dirac structure plays an 
essential role in the definition of implicit Lagrangian systems (or, alternatively, Lagrange-Dirac systems).
\medskip

Let $\Delta_{Q} \subset TQ$ be a regular distribution on $Q$ and define a lifted distribution on $T^{\ast}Q$ by
\begin{equation*}
\Delta_{T^{\ast}Q}
=( T\pi_{Q})^{-1} \, (\Delta_{Q}) \subset TT^{\ast}Q,
\end{equation*}
where $\pi_{Q}:T^{\ast}Q \to Q$ is the canonical projection so that its tangent is a map
$T\pi_{Q}:TT^{\ast}Q \to TQ$. Let $\Omega_{T^*Q}$ be the canonical two-form on $T^{\ast}Q$.  
The {\it induced Dirac structure} $D_{\Delta_Q}$ on $T^{\ast}Q$, is the subbundle of $T  T^{\ast}Q \oplus T ^{\ast} T^{\ast}Q$, whose fiber is given for each $p_{q} \in
T^{\ast}Q$ as
\begin{equation}\label{InducedDirac}
\begin{split}
D_{\Delta_Q}(p_{q})=\{ (v_{p_{q}}, \alpha_{p_{q}})
\in T_{p_{q}}T^{\ast}Q \times T^{\ast}_{p_{q}}T^{\ast}Q  & \mid
v_{p_{q}} \in
\Delta_{T^{\ast}Q}(p_{q})  \;\mbox{and} \;   \\ 
&\alpha_{p_{q}}- \Omega^{\flat}(p _q) (v_{p_{q}}) \in \Delta^{\circ}_{T^{\ast}Q}(p_{q}) \}.
\end{split}
\end{equation}

\paragraph{Local Representation of the Dirac Structure.}
Let $q^i$ be local coordinates on $Q$ so that locally, $Q$ is represented by an open set $U \subset
\mathbb{R}^n$. The constraint set $\Delta_Q$  defines a subspace of $TQ$,
which we denote by $\Delta(q) \subset \mathbb{R}^n$ at each point
$q \in U $. If we let the dimension of the constraint space be $n-m$,
then we can choose a basis $e _{m+1}(q), e _{m+2}(q),\ldots, e _n (q)$ of
$\Delta(q)$.

It is also common to represent constraint sets as the simultaneous kernel of a number of
constraint one-forms; that is, the annihilator of $\Delta(q)$, which is denoted by 
$\Delta^{\circ}(q)$, is
spanned by such one-forms, that we write as
$\mu^{1}, \mu^{2}, \ldots, \mu^{m}$.
Now writing the projection map $\pi_Q: T^{\ast}Q\rightarrow Q $
locally as $(q,p) \mapsto q$, its tangent map is locally given by
$T\pi_Q : (q, p, \dot{q}, \dot{p}) \mapsto (q, \dot{q})$. Thus, we
can locally represent $\Delta_{T^{\ast}Q}$ as
\[
\Delta_{T^{\ast}Q} \cong \left\{ (q,p, \dot{q}, \dot{p} )
\mid q \in U, \dot{q} \in \Delta (q) \right\}.
\]
Let us denote a point in $T ^{\ast} T ^{\ast} Q $ by
$(q,p, \alpha, w )$, where $\alpha$ is a covector
and $w $ is a vector, notice that the annihilator of $\Delta_{T^{\ast}Q}$ is
locally,
\[
\Delta^{\circ}_{T^{\ast}Q} \cong \left\{ (q,p,
\alpha, w )
\mid q \in U,\, \alpha \in \Delta^{\circ} (q) \; \mbox{and} \; w
= 0
\right\}.
\]
Recall the symplectomorphism $\Omega^{\flat}_{T^*Q}:TT^*Q\Flder T^*T^*Q$ is given in local by
\[
\Omega^\flat_{T^*Q} (q,p)(\dot{q}, \dot{p} ) = (- \dot{p}, \dot{q}).
\]
Thus, it follows from equation \eqref{InducedDirac} that the local expression of the induced Dirac structure is given by
\begin{align}\label{localdirac}
D_{\Delta_Q}(q,p)  =
\left\{
\left( (q,p, \dot{q}, \dot{p}), (q,p, \alpha, w) \right) \mid
\dot{q} \in \Delta (q), \; w = \dot{q},  \; \mbox{and} \;
\alpha +\dot{p} \in \Delta^{\circ}(q)
\right\}.
\end{align}


\paragraph{Lagrange-Dirac Dynamical Systems.} Following \cite{YoMa2006a, YoMa2006b}, we shall briefly see the theory of Lagrange-Dirac systems. 
Let $L:TQ \to \mathbb{R}$ be a Lagrangian, possibly degenerate. The differential $ \mathbf{d} L:TQ \rightarrow T^{\ast}TQ$ of $L$ is the one-form on $TQ$ locally given by
\[
\mathbf{d} L(q,v)= \left( q,v, \frac{\partial L}{\partial q}, \frac{\partial L}{\partial v}\right) .  
\]
Using the canonical diffeomorphism $ \gamma_{Q}:T^{\ast}TQ \rightarrow T^{\ast}T^{\ast}Q$, we define the {\it Dirac differential} of $L$ by 
\[
\mathbf{d}_{D} L:= \gamma_{Q} \circ \mathbf{d} L,
\]
which is locally given by
$$
\mathbf{d}_{D} L(q,v)= \left(q,\frac{\partial L}{\partial v}, - \frac{\partial L}{\partial q},  v \right),
$$
where $(q,v)$ are local coordinates for $TQ$, $(q, p)$ for $T^*Q$ and $(q,v,p)$ for $TQ\oplus T^*Q$.

\begin{definition}[\bf Lagrange-Dirac dynamical systems]
The equations of motion of a {\rm Lagrange-Dirac dynamical system} (or an implicit Lagrangian system)  $(Q, \Delta_Q, L)$ are given by 
\begin{equation}\label{LDirac_system} 
\left( (q(t),p(t),\dot{q}(t),\dot{p}(t)), \mathbf{d}_{D} L(q(t),v(t))\right)  \in D_{ \Delta _Q }(q(t),p(t)).
\end{equation} 
Any curve $(q(t),v(t),p(t)) \in TQ \oplus T^{\ast}Q,\,t_{1} \le t \le t_{2}$ satisfying \eqref{LDirac_system} is called a {\rm solution curve} of the Lagrange-Dirac dynamical system.
\end{definition}
It follows from \eqref{localdirac}  and \eqref{LDirac_system} that $(q(t),v(t),p(t))$, $t _1 \leq t \leq t _2 $ is a solution curve if and only if it satisfies the implicit Lagrange-d'Alembert equations
\begin{equation}\label{ExtLagDAEqn_Local}
p =\frac{\partial L}{\partial v }, \quad \dot{q} =v \in \Delta(q), \quad  \dot{p} - \frac{\partial L}{\partial q}
\in \Delta^{\circ}(q).
\end{equation}
Notice that the equations \eqref{ExtLagDAEqn_Local} are equal to the nonholonomic equations in \eqref{NonHEq}.

\begin{remark}{\rm Note that the equation $p ={\partial L}/{\partial v }$ arises from the equality of the base points $(q,p)$ and $\left(q,{\partial L}/{\partial v}\right)$ in \eqref{LDirac_system}. }
\end{remark}

\paragraph{Energy Conservation for Implicit Lagrangian Systems.}
Let  $(Q, \Delta_{Q}, L)$ be a  Lagrange-Dirac dynamical system.
Define the generalized energy function $E$ on $TQ \oplus T^{\ast}Q$ by
\[
E(q,v,p)=\langle p, v \rangle -L(q,v).
\]
If $(q(t), v(t),p(t))$ in $TQ \oplus T^{\ast}Q$ is a solution curve of the Lagrange-Dirac system $(Q, \Delta_{Q}, L)$, then the energy $E(q(t),v(t),p(t))$ is constant along the solution curve. This is shown as follows:
\[
\frac{d}{dt} E = \left\langle \dot{p}, v \right\rangle
                   + \left\langle p, \dot{v} \right\rangle
                   - \left<\frac{\partial L }{\partial q }, \dot{q}\right>
                   - \left<\frac{\partial L }{\partial v }, \dot{v}\right>  = \left\langle
       \dot{p} - \frac{\partial L}{\partial q}, v
       \right\rangle,
\]
which vanishes since $\dot{q}=v \in \Delta (q)$ and
since $\dot{p} - {\partial L}/{\partial q} \in \Delta^{\circ} 
(q)$.

\paragraph{The Lagrange-d'Alembert-Pontryagin Principle.(\cite{YoMa2006b})}
Let us see how the Lagrange-Dirac dynamical system can be developed from the {\it Lagrange-d'Alembert-Pontryagin
principle}  for a curve $(q(t),v(t),p(t))$, $t_{1} \le t \le t_{2},$  in $TQ \oplus T^{\ast}Q$, which is given by
\begin{equation}\label{LagDAPontPrin_Force}
\begin{split}
&\delta \int_{t_1}^{t_2} \biggl[ L(q(t),v(t)) + \left<p(t),\,\dot{q}(t)-v(t)\right> \biggr]\,dt \\
&=\delta \int_{t_1}^{t_2} \biggl[ \left<p(t), \,\dot{q}(t) \right>-E(q(t), v(t), p(t)) \biggr]\,dt\\
&=0
\end{split}
\end{equation}
for chosen variations $\delta{q}(t) \in \Delta_Q(q(t))$ and with the constraint $v(t) \in  \Delta_Q(q(t))$. Keeping the endpoints of $q(t)$ fixed, we have
\begin{equation}\label{LagDAPontPrin_Force_Local}
\begin{split}
& \int_{t_1}^{t_2}\biggl[ 
\left< \frac{\partial L}{\partial q}-\dot{p}, \delta{q} \right>+ \left< \frac{\partial
L}{\partial{v}} -p,\, \delta{v} \right> + \left<\delta{p}, \, \dot{q}-v \right>
\biggr] \,dt  =0
\end{split}
\end{equation}
for variations $\delta{q}(t) \in \Delta_Q(q(t))$, for all $\delta{v}(t)$ and $\delta{p}(t)$, and with $v(t) \in \Delta_Q(q(t))$.

Thus, we obtain the local expression of the Lagrange-Dirac dynamical system in \eqref{ExtLagDAEqn_Local} from the  Lagrange-d'Alembert-Pontryagin
principle in \eqref{LagDAPontPrin_Force_Local}.

For the case in which $\Delta_{Q}=TQ$, this recovers the {\it Hamilton-Pontryagin principle}, which induces the {\it implicit Euler-Lagrange equations}.

\paragraph{Dirac Structures in Hamiltonian Systems.}
In general, the standard Dirac structure $D_{P}$ on a manifold $P$ is given by the graph of a skew-symmetric bundle map over $P$ as shown above. The Hamilton-Dirac system can be given by a pair $(D_{P},H)$ that satisfies, for each $z\in P$,
\begin{equation}\label{canonicalD}
(\dot z, \mathbf{d}H(z)) \in D_{P}(z),
\end{equation}
where $H: P \to \mathbb{R}$ denotes a Hamiltonian. 
 The idea was shown by  \cite{VaMa1995} to the case of a nontrivial distribution $\Delta_{P} \subset TP$ on a Poisson manifold $P$, which is called {\it implicit Hamiltonian systems}.

For the case $P=T^{\ast}Q$, using the usual local coordinates $z=(q,p)$ leads to the canonical Dirac structure and the standard Hamiltonian system. The standard Hamiltonian equations can be also formulated by Hamilton's phase space principle (see \cite{YoMa2006b}):
$$
\delta \int_{t_1}^{t_2}\biggl[\left< p, \dot{q}\right>-H(q,p)\biggr]\,dt=0
$$
with the fixed endpoint conditions $\delta q(t_1)= \delta q(t_2)=0$. 

Intrinsically, Hamilton's phase space principle is described by
\[
\delta \int_{t_1}^{t_2}\biggl[\bra\Theta_P(z),\,\dot z\ket-H(z)\biggr] dt=0,
\]  
where $\Theta_P$ is a one-form on $P$, which may induce under the condition of the variation of the curves fixed at the endpoints:
\[
\mathbf{i}_{\dot z}\Omega_P(z)=\de H(z),
\]
where $\Omega_P=-\de\Theta_P$. The above construction is clearly consisted with the construction using the canonical Dirac structure as in \eqref{canonicalD}.

%

\section{Lagrange-Dirac Systems in Vakonomic Mechanics}

\paragraph{The Hamilton-Pontryagin Principle for Vakonomic Lagrangians.}
Let us consider the Hamilton-Pontryagin principle for vakonomic Lagrangians. To do this, let $L: TQ \to \mathbb{R}$ be a Lagrangian, possibly degenerate, and consider the following nonholonomic constraints:
$$
\phi^{\alpha}(q,v)=0,\quad \alpha=1,...,m<n;
$$
where $(q,v)$ are the local coordinates of $TQ$.
\begin{definition}
Define a vakonomic Lagrangian $\ext:TQ\times V^{\ast}\Flder\R$ by
\begin{equation}\label{VakLag}
\ext(q,v,\lambda):=L(q,v)+\lambda_{\alpha}\,\phi^{\alpha}(q,v),
\end{equation}
where $L:TQ\Flder\R$ is the usual Lagrangian and we consider $\lambda_{\alpha}$ as the local coordinates of the dual vector space $V^*$. Define also the {\bfi vakonomic Lagrangian energy} $E_{\ext}:(TQ \oplus T^{\ast}Q)\times V^{\ast}\Flder\R$ by 
\begin{equation*}\label{vakLagEn}
E_{\ext}(q,v,p,\lambda):=\left<p,v\right>-\ext(q,v,\lambda),
\end{equation*}
where $(q,v,p,\lambda)$ are local coordinates of the {\bfi vakonomic Pontryagin bundle} $(TQ \oplus T^{\ast}Q)\times V^{\ast}$.
\end{definition}

\begin{proposition}\label{Lagrangian0}
Let $\ext:TQ\times V^{\ast}\Flder\R$ be a (possibly degenerate) Lagrangian function. Define the action functional
\begin{equation}\label{action1}
\begin{split}
&\int_{t_1}^{t_2}\biggl[ \ext(q(t),v(t),\lambda(t))+ \left<p(t), \dot q(t)-v(t) \right> \biggr] dt\\
&\qquad =\int_{t_1}^{t_2}\biggl[  \left<p(t), \dot q(t)\right> -E_{\ext}(q,v,p,\lambda)\biggr]\,dt.
\end{split}
\end{equation}
Keeping the endpoints of $q(t)$ fixed, whereas the endpoints of $v(t)$, $p(t)$ and $\lambda(t)$ are allowed to be free, the stationary condition for this action functional induces the {\bfi local implicit vakonomic Euler-Lagrange equations}:
\begin{align}\label{vakoELeq}
p=\frac{\der\ext}{\der v},\quad
\dot q=v,\quad
\dot{p}=\frac{\der\ext}{\der q},\quad
0=\frac{\der\ext}{\der\lambda},
\end{align}
which are restated by
\begin{equation}\label{ImVDynamics}
\begin{split}
p=\frac{\der L}{\der v}+\lambda_{\alpha}\frac{\der\phi^{\alpha}}{\der v},\quad \dot q=v,\quad 
\dot p=\frac{\der L}{\der q}+\lambda_{\alpha}\frac{\der\phi_{\alpha}}{\der q},\quad
\phi^{\alpha}(q,\dot q)=0.
\end{split}
\end{equation}
Notice that the above equations are equivalent with \eqref{VDynamics}.

\end{proposition}
\begin{proof}
By direct computations, the variation of the action functional \eqref{action1} is given by 
\begin{equation*}
\begin{split}
&\delta\int_{t_1}^{t_2}\biggl[\ext(q(t),v(t),\lambda(t))+ \left<p(t), \dot q(t)-v(t) \right>\biggr]\,dt=
\int_{t_1}^{t_2}\left[\left<\frac{\der\ext}{\der q}-\dot p,\,\delta\,q\right> \right. \\
&\hspace{3cm}\left.
+\left<\frac{\der\ext}{\der v}- p,\,\delta\,v\right>
+\left<\frac{\der\ext}{\der\lambda},\delta\lambda\right>+\left<\dot q-v,\,\delta\,p\right>
\right] \,dt+\left< p,\delta q\right>\bigg|_{t_1}^{t_2},
\end{split}
\end{equation*}
where integration by parts has been taken into account. Keeping the endpoints of $q(t)$ fixed, namely, $\delta q(t_1)=\delta q(t_2)=0$, the stationarity condition for the action functional with free variations $(\delta q,\delta v,\delta\lambda,\delta p)$ provides the set of equations \eqref{vakoELeq}, which lead to \eqref{ImVDynamics} straightforwardly from the definition of $\ext$.
\end{proof}

We call the above variational principle as the {\bfi Hamilton-Pontryagin principle for the vakonomic Lagrangian} $\ext(q,v,\lambda)$.

\paragraph{The Intrinsic Implicit Vakonomic Euler-Lagrange Equations.} Our next purpose is to develop an {\it intrinsic form for the vakonomic Euler-Lagrange equations}. 

Let $\Theta_{T^*Q}$ be the canonical one-form on $T^{\ast}Q$ and thus $\Omega_{T^*Q}=-\mathbf{d}\Theta_{T^*Q}$ is the canonical two-form on $T^{\ast}Q$.
Define the projections
\begin{equation*}
\begin{split}
\overline{\mathrm{pr}}_{T^{\ast}Q}&:(TQ \oplus T^{\ast}Q)\times V^{\ast} \to T^{\ast}Q;\quad (q,v,p,\lambda) \mapsto (q, p),\\
\overline{\mathrm{pr}}_{Q}&:(TQ \oplus T^{\ast}Q) \times V^{\ast} \to Q;\quad \,\,\,\,\,\,\,(q,v,p,\lambda) \mapsto q.
\end{split}
\end{equation*}
One can define a pre-symplectic form $\overline{\Omega}$ on $(TQ \oplus T^{\ast}Q)\times V^{\ast}$ by
\begin{eqnarray}\label{PreSymp}
\overline{\Omega}:=\overline{\mathrm{pr}}_{T^{\ast}Q}^{\ast}\Omega_{T^*Q}.
\end{eqnarray}
In the above, notice that $\overline{\Omega}=-\mathbf{d}\overline{\Theta}$ holds since $\overline{\Theta}:=\overline{\mathrm{pr}}_{T^{\ast}Q}^{\ast}\Theta_{T^*Q}$ is the one-form on $(TQ \oplus T^{\ast}Q)\times V^{\ast}$. Thus, it follows
$$
\overline{\Theta}(q,v,p,\lambda)=p\,dq,\qquad \overline{\Omega}(q,v,p,\lambda)=dq \wedge dp.
$$

\begin{definition}
Let $x(t)=(q(t),v(t),p(t),\lambda(t)),\; t \in [t_{1},t_{2}]$ be a curve in $(TQ \oplus T^{\ast}Q)\times V^{\ast}$.
Let us define the action functional for $x(t)$ by
\begin{equation}\label{Intrinsic}
\int_{t_1}^{t_2}\biggl[\,\bra\overline{\Theta}(x(t))\,,\, \dot{x}(t)\ket-E_{\ext}(x(t)) \biggr]\,dt,
\end{equation}
which is the intrinsic expression of \eqref{action1}.
\end{definition}
\begin{proposition}\label{larga}
Under the endpoints of $q(t)=\overline{\mathrm{pr}}_{Q}(x(t))$ fixed, the stationarity condition of the action functional \eqref{Intrinsic} singles out a critical curve $x(t)$ that satisfies the {\bfi intrinsic implicit vakonomic Euler-Lagrange equations:}
\begin{eqnarray}\label{IntImpEulLagEqnj}
\mathbf{i}_{\dot x(t)}\overline{\Omega}(x(t))=\mathbf{d}E_{\ext}(x(t)).
\end{eqnarray}
\end{proposition}
\begin{proof}
The stationarity condition of the action functional is given by
\begin{eqnarray*}\label{ActFunVakoLag}
&&\delta \int_{t_1}^{t_2}\biggl[\bra\overline{\Theta}(x(t)),\, \dot{x}(t)\ket-E_{\ext}(x(t))\biggr] \,dt\nonumber\\
&&\;\;=\int_{t_1}^{t_2}\left< \mathbf{i}_{\dot x(t)}\overline{\Omega}(x(t))-\mathbf{d}E_{\ext}(x(t)), \delta x(t) \right>\,dt+ \bra\overline{\Theta}(x(t))\,,\, \delta{x}(t)\ket\bigg|_{t_1}^{t_2}\\[2mm]
&&\;\;=0, \nonumber
\end{eqnarray*}
for all variations $\delta x(t)=(\delta q(t),\delta v(t),\delta p(t),\delta \lambda(t))$ with the endpoints of $q(t)=\overline{\mathrm{pr}}_{Q}(x(t))$ fixed.
Thus, one obtains equation \eqref{IntImpEulLagEqnj}
\medskip

In fact, the left-hand side of \eqref{IntImpEulLagEqnj} is locally given by
\begin{eqnarray*}
\mathbf{i}_{\dot x(t)}\overline{\Omega}(x(t))=\left( -\dot{p} \right)dq+\dot{q}\,dp
\end{eqnarray*}
and the right-hand side is denoted by
\begin{eqnarray*}
\de E_{\ext}(q,v,p,\lambda)&=&\left<\frac{\der E_{\ext}}{\der q},\,dq \right>+\left<\frac{\der E_{\ext}}{\der v},\,dv\right>+ \left<dp, \frac{\der E_{\ext}}{\der p}\right>+\left<d\lambda,\, \frac{\der E_{\ext}}{\der \lambda} \right> \nonumber\\
&=&\left<-\frac{\der\ext}{\der q},\,dq\right>+\left< p-\frac{\der\ext}{\der v},\,dv \right>+\left<dp,\,v \right>+\left<d\lambda,\,\frac{\der\ext}{\der \lambda}\right>. \label{dE}
\end{eqnarray*}
Thus, the equation \eqref{IntImpEulLagEqnj} leads to the local expression of the implicit vakonomic Euler-Lagrange equations given in \eqref{vakoELeq}.
\end{proof}

We call the above variational principle the {\bfi Hamilton-Pontryagin principle for the vakonomic Lagrangian}. 

\paragraph{The Lagrange-Dirac Dynamical System on $(TQ \oplus T^{\ast}Q) \times V^{\ast}$.} 
Recall that we can naturally define a presymplectic form $\overline{\Omega}$ on $(TQ \oplus T^{\ast}Q) \times V^{\ast}$ as in \eqref{PreSymp}. Then, we can also define the associated bundle map 
$$
\overline{\Omega}^{\flat}: T\left((TQ \oplus T^{\ast}Q) \times V^{\ast}\right)\to T^{\ast}\left((TQ \oplus T^{\ast}Q) \times V^{\ast} \right)
$$
by, for $x \in (TQ \oplus T^{\ast}Q) \times V^{\ast}$,
$$
\overline{\Omega}^{\flat}(x)\cdot \dot{x}=\mathbf{i}_{\dot x(t)}\overline{\Omega}(x).
$$

\begin{definition}
Define the Dirac structure on $(TQ \oplus T^{\ast}Q) \times V^{\ast}$ by using the pre-symplectic two-form \eqref{PreSymp} as
$$
\overline{D}=\mathrm{graph}\,\overline{\Omega}^{\flat} \subset T\left((TQ \oplus T^{\ast}Q) \times V^{\ast}\right) \oplus T^{\ast}\left((TQ \oplus T^{\ast}Q) \times V^{\ast} \right).
$$
\end{definition}
\begin{proposition}\label{larga}
The equations of motion of the {\rm vakonomic Lagrange-Dirac dynamical system} $(\overline{D},E_{\ext})$ are given by, for each $x \in (TQ \oplus T^{\ast}Q) \times V^{\ast}$,
\begin{eqnarray}\label{VakoLagDirac}
\left( \dot{x}, \mathbf{d}E_{\mathfrak{L}}(x)\right) \in \overline{D}(x).
\end{eqnarray}
Using local coordinates $x=(q,v,p,\lambda) \in (TQ \oplus T^{\ast}Q) \times V^{\ast}$, equation \eqref{VakoLagDirac} induces the local implicit vakonomic Euler-Lagrange
equations in \eqref{vakoELeq}.

\end{proposition}
\begin{proof}
The Dirac structure $\overline{D}$ is locally denoted, for each $x=(q,v,p,\lambda) \in (TQ \oplus T^{\ast}Q) \times V^{\ast}$, by
\begin{align}\label{local-overlinedirac}
\overline{D}(x)  & = \biggl\{\left((\dot{q},\dot{v},\dot{p},\dot{\lambda}),(\alpha,\beta,w,u) \right) 
\mid  
w = \dot{q},  \; \beta=0, \; \alpha +\dot{p}=0,\; \mbox{and},\; u=0 \biggr\},
\end{align}
where $(\alpha,\beta,w,u)\in T^*_x((TQ \oplus T^{\ast}Q) \times V^{\ast})$. The particular local form of \eqref{local-overlinedirac} follows directly from $\overline{\Omega}=dq\wedge dp$, more concretely $\mathbf{i}_{\dot x}\overline{\Omega}=(-\dot p,0,\dot q,0)$. Consequently, it follows from \eqref{dE} and \eqref{local-overlinedirac} that one can obtain the implicit vakonomic Euler-Lagrange
equations in \eqref{vakoELeq} when setting $(\alpha,\beta,w,u)=\de E_{\mathfrak{L}}(q,v,p,\lambda)$, namely
\[
\de E_{\mathfrak{L}}(q,v,p,\lambda)=\lp\frac{\der E_{\mathfrak{L}}}{\der q}, \frac{\der E_{\mathfrak{L}}}{\der v},\frac{\der E_{\mathfrak{L}}}{\der p}, \frac{\der E_{\mathfrak{L}}}{\der\lambda}\rp.
\]
\end{proof}
\paragraph{The Vakonomic Dirac Differential Operator.} In the previous paragraph we have defined the vakonomic Lagrange-Dirac dynamical system in terms of the pair $(\overline{D}, E_{\ext})$. Here, by analogy with the construction of implicit Lagrangian systems, we shall define an alternative notion of {\it vakonomic Lagrange-Dirac system}. To do this, we introduce the following isomorphisms:
$$
\iota_{1}: T^{\ast}(TQ \times V^{\ast}) \to T^{\ast}TQ \times T^{\ast}V^{\ast}; \;\; (q,\delta{q},\lambda,\delta{p},p,w) \mapsto (q,\delta{q},\delta{p},p,\lambda,w),
$$
and
$$ 
\iota_{2}:T^{\ast}(T^{\ast}Q \times V^{\ast}) \to T^{\ast}T^{\ast}Q \times T^{\ast}V^{\ast}; \;\; (q,p,\lambda,-\delta{p},\delta{q},w) \mapsto (q,p,-\delta{p},\delta{q},\lambda,w),
$$ 
where $(\lambda, w)\in T^*V^*$. Then, we define a diffeomorphism  $\tilde{\gamma}_Q$ between $T^{\ast}(T^{\ast}Q \times V^{\ast})$ and $T^{\ast}(TQ \times V^{\ast})$ as
\begin{equation}\label{tildegamma}
\begin{split}
\tilde{\gamma}_Q:= \iota_{2}^{-1} \circ ({\gamma}_Q \times \mathrm{Id}) \circ \iota_{1} &: T^{\ast}(TQ \times V^{\ast}) \to T^{\ast}(T^*Q \times V^{\ast});\\[1mm]
& (q,\delta{q},\lambda,\delta{p},p,w) \mapsto (q,p,\lambda,-\delta{p},\delta{q},w),
\end{split}
\end{equation}
where $\mathrm{Id}: T^{\ast}V^{\ast} \to T^{\ast}V^{\ast}$ is the identity map.

Given a vakonomic Lagrangian $\mathfrak{L}(q,v,\lambda)$ on  $TQ\times V^{\ast}$, its differential  is a one-form $\de\ext:TQ\times V^{\ast}\Flder T^{\ast}(TQ\times V^{\ast})$, which may be locally described by
\[
\de\ext(q,v,\lambda)=\lp q,v,\lambda,\frac{\der\ext}{\der q},\frac{\der\ext}{\der v},\frac{\der\ext}{\der\lambda}\rp.
\]
Using the diffeomorphism $\tilde\gamma_Q$ in \eqref{tildegamma}, we can define a differential operator $\Dirac$ called the {\it vakonomic Dirac Differential} of $\ext$ by
\begin{equation}
\Dirac\ext:=\tilde\gamma_Q\circ\de\,\ext.
\end{equation}
Namely, the map $\Dirac\ext:TQ\times V^*\Flder T^*(T^*Q\times V^*)$  is locally given by
\begin{equation}\label{DiracLocal}
\Dirac\ext(q,v,\lambda)=\lp q,\frac{\der\ext}{\der v},\lambda,-\frac{\der\ext}{\der q},v,\frac{\der\ext}{\der \lambda}\rp.
\end{equation}

\paragraph{Pre-symplectic Form on $T^{\ast}Q \times V^{\ast}$}
Using the natural projection
$$
\widehat{\mathrm{pr}}_{T^{\ast}Q}:T^{\ast}Q \times V^{\ast} \to T^{\ast}Q;\quad (q,p,\lambda) \mapsto (q,p),
$$
we can define a pre-symplectic structure $\widehat{\Omega}$ on $T^{\ast}Q \times V^{\ast}$ by
\begin{equation}\label{pre-symp}
\widehat{\Omega}:=\widehat{\mathrm{pr}}_{T^{\ast}Q}^{\ast}\Omega_{T^*Q},
\end{equation}
which is given in local form, for each $(q,p,\lambda) \in T^{\ast}Q \times V^{\ast}$, by
$$
\widehat{\Omega}(q,p,\lambda)\left((\dot{q},\dot{p},\dot{\lambda}),(\delta{q},\delta{p},\delta{\lambda})\right)=\left<\delta{p}, \dot{q} \right>-
\left<\dot{p}, \delta{q} \right>.
$$
Associated with $\widehat{\Omega}$, we have the bundle map
$$
\widehat{\Omega}^{\flat}: T(T^{\ast}Q \times V^{\ast}) \to T^{\ast}(T^{\ast}Q \times V^{\ast}),
$$
which is locally denoted by
$$
(q,p,\lambda,\delta{q},\delta{p},\delta{\lambda}) \mapsto (q,p,\lambda,-\delta{p},\delta{q},0).
$$

\paragraph{Dirac Structure on $T^{\ast}Q\times  V^{\ast}$.} Before constructing the Lagrange-Dirac system for the vakonomic mechanics, we shall define a Dirac structure on $T^{\ast}Q\times V^{\ast}$ as in the below.

\begin{definition}
Define a Dirac structure $\widehat{D}$ on $T^{\ast}Q\times  V^{\ast}$ by using the pre-symplectic form $\widehat{\Omega}$ in \eqref{pre-symp} as 
$$
\widehat{D}=\mathrm{graph}\,\widehat{\Omega}^{\flat} \subset T\left(T^{\ast}Q \times V^{\ast}\right) \oplus T^{\ast}\left(T^{\ast}Q \times V^{\ast} \right).
$$
\end{definition}
\begin{proposition}
The Dirac structure $\widehat{D}$ on $T^{\ast}Q\times  V^{\ast}$ is locally denoted, for each $(q,p,\lambda) \in T^{\ast}Q \times V^{\ast}$, by
\begin{align}\label{local-widehatdirac}
\widehat{D}(q,p,\lambda) & = \biggl\{\left((\dot{q},\dot{p},\dot{\lambda}),(\alpha,u,w) \right) 
\mid  
u= \dot{q},  \;  \alpha +\dot{p}=0,\; \mbox{and},\; w=0 \biggr\},
\end{align}
where we denote $(\alpha,u,w)\in T^*(TQ\times V^*).$
\end{proposition}
\begin{proof}
Again, the claim follows from the particular local form of $\widehat{\Omega}$ given by \eqref{pre-symp}, namely $\widehat{\Omega}=dq\we dp$. Therefore, if 
$\lp(\dot{q},\dot{p},\dot{\lambda}),(\alpha,u,w)\rp\in\widehat{D}(q,p,\lambda)$ it follows that
\[
\widehat{\Omega}\lp(\dot{q},\dot{p},\dot{\lambda}),\cdot\rp=(-\dot p,\dot q,0)=(\alpha,u,w),
\]
which finishes the proof.
\end{proof}

\paragraph{The Vakonomic Lagrange-Dirac Systems on $T^*Q\times V^*$.}
We give the definition of {\it vakonomic Lagrange-Dirac systems on $T^{\ast}Q\times  V^{\ast}$} as follows:
\begin{definition}
Let $\ext:TQ\times V^{\ast}\Flder\R$ be a given Lagrangian function (possibly degenerate). Let $(q(t),v(t),p(t),\lambda(t)),\, t\in [t_{1},t_{2}]$, be a curve in $(TQ \oplus T^{\ast}Q) \times  V^{\ast}$. The equations of motion for the {\bfi vakonomic Lagrange-Dirac system}  $(\widehat{D},\mathfrak{L})$ are given by
\begin{equation}\label{ImpVakLagSys2}
((\dot{q}(t),\dot{p}(t),\dot{\lambda}(t)),\mathbf{d}_{D}\mathfrak{L}(q(t),v(t),\lambda(t)))\in \widehat{D}(q(t),p(t),\lambda(t)).
\end{equation}
\end{definition}
\begin{proposition}
The curve $(q(t),v(t),p(t),\lambda(t))$ is a solution curve of the vakonomic Lagrange-Dirac system on $T^{\ast}Q\times  V^{\ast}$ in \eqref{ImpVakLagSys2}  if and only if it verifies 
\begin{equation}\label{ImpVakEul-Lag2}
\mathbf{i}_{(\dot{q}(t),\dot{p}(t),\dot{\lambda}(t))}\widehat\Omega(q(t),p(t),\lambda(t)) =\de_{D}\mathfrak{L}(q(t),v(t),\lambda(t))),
\end{equation}
which are locally denoted by  equations in \eqref{vakoELeq}. 
\end{proposition}
\begin{proof}
By definition, we have
$$
\widehat{D}=\mathrm{graph}\,\widehat{\Omega}^{\flat}
$$
and hence it follows from \eqref{ImpVakLagSys2} that, for the solution curve $(q(t),v(t),p(t),\lambda(t))$,
$$
\widehat{\Omega}^{\flat}(q,p,\lambda)\left( \dot{q},\dot{p},\dot{\lambda}\right)=\mathbf{d}_{D}\mathfrak{L}(q,v,\lambda).
$$
Since $\widehat{\Omega}^{\flat}(q,p,\lambda)(\dot{q},\dot{p},\dot{\lambda}) := \mathbf{i}_{(\dot{q},\dot{p},\dot{\lambda})}\widehat\Omega(q,p,\lambda)$ on the left-hand side of the above equation, we obtain the  implicit vakonomic Euler-Lagrange equations \eqref{ImpVakEul-Lag2}.
\medskip

Locally, we have 
$$
\widehat{\Omega}^{\flat}(q,p,\lambda)\lp\dot{q},\dot{p},\dot{\lambda}\rp=(-\dot p,\dot q,0), 
$$
while we recall that 
$$
\mathbf{d}_{D}\mathfrak{L}(q,v,\lambda)=\left(-\frac{\der\ext}{\der q},v,\frac{\der\ext}{\der \lambda}\right), 
$$
such that the base point holds:
$$
(q,p)=\left(q, \frac{\der\ext}{\der v} \right).
$$
Therefore, we arrive to 
\[
p=\frac{\der\mathfrak{L}}{\der v},\,\,\,\,\,\,\,\dot p=\frac{\der\mathfrak{L}}{\der q},\,\,\,\,\,\,\,\dot q=v,\,\,\,\,\,\,\,\,\frac{\der\mathfrak{L}}{\der\lambda}=0,
\]
which are equations \eqref{vakoELeq}.
\end{proof}
Needless to say, equations \eqref{ImpVakEul-Lag2} are equal to the {\bfi implicit vakonomic Euler-Lagrange equations} in \eqref{IntImpEulLagEqnj}.
\medskip

Now we have described the vakonomic dynamics from several points of view; namely, the Hamilton-Pontryagin variational principle as well as the intrinsic Lagrange-Dirac dynamical systems using Dirac structures. Our results may be summarized in the following theorem:
\begin{theorem}\label{TEO}
The following statements are equivalent:
\begin{enumerate}
\item The Hamilton-Pontryagin principle for the following action integral
\[
\int_{t_1}^{t_2}\lc\ext(q(t),v(t),\lambda(t))+ \left<p(t), \dot q(t)-v(t) \right>\rc\,dt
\]
holds for any variations of $q(t)$ with fixed endpoints.
\item The curve $(q(t),v(t),p(t),\lambda(t))\in (TQ\oplus T^*Q)\times V^*,\, t \in[t_{1},t_{2}]$, satisfies the implicit vakonomic Euler-Lagrange equations
\begin{eqnarray*}\label{IntImpEulLagEqn}
\mathbf{i}_{\dot x(t)}\overline{\Omega}(x(t))=\mathbf{d}E_{\ext}(x(t)),
\end{eqnarray*}
which are locally given by
$$
p=\frac{\der\ext}{\der v},\quad
\dot q=v,\quad
\dot{p}=\frac{\der\ext}{\der q},\quad
0=\frac{\der\ext}{\der\lambda}.
$$

\item The curve $(q(t),v(t),p(t),\lambda(t)),\, t \in[t_{1},t_{2}]$ is a solution curve of the vakonomic Lagrange-Dirac dynamical system $(\overline{D},E_{\ext})$ which satisfies
\begin{eqnarray*}
\left( (\dot{q}(t),\dot{v}(t),\dot{p}(t),\dot{\lambda}(t)), \mathbf{d}E_{\mathfrak{L}}(q(t),v(t),p(t),\lambda(t))\right) \in \overline{D}(q(t),v(t),p(t),\lambda(t)).
\end{eqnarray*}

\item The curve  $(q(t),v(t),p(t),\lambda(t))$ is a solution curve of the vakonomic Lagrange-Dirac dynamical system $( \widehat{D},\mathfrak{L})$ which satisfies

\begin{equation*}
((\dot{q}(t),\dot{p}(t),\dot{\lambda}(t)),\mathbf{d}_{D}\mathfrak{L}(q(t),v(t),\lambda(t)))\in \widehat{D}(q(t),p(t),\lambda(t)),
\end{equation*}
which is equivalent to
\begin{eqnarray*}
\mathbf{i}_{(\dot{q}(t),\dot{p}(t),\dot{\lambda}(t))}\widehat\Omega(q(t),p(t),\lambda(t)) =\de_{D}\mathfrak{L}(q(t),v(t),\lambda(t))).
\end{eqnarray*}
\end{enumerate}
\end{theorem}

\section{Examples}

\paragraph{The Vertical Rolling Disk.} Let us consider the following problem for a disk of
radius $R$ and unit mass $m=1$ which rolls on a horizontal plane. The configuration space for this system can be identified with $Q=\R^2\times S^1\times S^1$. By $(x,y)$ we denote the coordinates of the point of contact of the disk with the plane and $(\theta,\varphi)$ give, respectively, the angle between the disk and the $x$ axis, and the angle of rotation
between a fixed diameter in the disk and the $y$ axis. Therefore, we will use the coordinate notation $q=(x,y,\theta,\varphi) \in Q$.

Given the endpoints of $q(t)$ fixed, we want to find the trajectories of the disk connecting such points that minimize the energy consumption. Assume that the disk rolls on a plane without slipping, which is given by the following nonholonomic constraints: 
\begin{eqnarray*}
\phi^{1}(x,y,\theta,\varphi,\dot{x}, \dot{y}, \dot{\theta}, \dot{\varphi})&=&\dot{x}\,\sin{\theta}-\dot{y}\,\cos{\theta}=0,\\
\phi^{2}(x,y,\theta,\varphi,\dot{x}, \dot{y}, \dot{\theta}, \dot{\varphi})&=&\dot{x}\,\cos{\theta}+\dot{y}\,\sin{\theta}-R\,\dot{\varphi}=0.
\end{eqnarray*}
As in $\S$\ref{octh}, this is considered as an optimal control problem by setting $B=Q$, $N=TQ$ and $\pi:TQ\Flder Q$. Using local coordinates $(q,v)=(x,y,\theta,\varphi,v_x,v_y,v_{\theta},v_{\varphi})$, the cost function is given by the following  Lagrangian  $L:TQ\Flder\R$: 
\begin{equation*}\label{LagEx}
L(x,y,\theta,\varphi,v_x,v_y,v_{\theta},v_{\varphi})=\frac{1}{2}(v_x^2+v_y^2+I_1v_{\theta}^2+I_2v_{\varphi}^2),
\end{equation*}
where $I_1$ and $I_2$ denote the momenta of inertia.

In fact, in this framework we regard the velocities as the {\it control} variables. Solving this optimal control problem is precisely the same as the vakonomic problem associated to the vertical rolling disk for the vakonomic Lagrangian on $TQ\times V^{\ast}$ by incorporating the nonholonomic constraints as
\begin{eqnarray*}
&&\ext(x,y,\theta,\varphi,v_x,v_y,v_{\theta},v_{\varphi},\lambda_{1}, \lambda_2)=L(x,y,\theta,\varphi,v_x,v_y,v_{\theta},v_{\varphi})\nonumber\\
&&\hspace{2cm}+\lambda_1\,\phi^{1}(x,y,\theta,\varphi,v_x,v_y,v_{\theta},v_{\varphi})+\lambda_2\,\phi^{2}(x,y,\theta,\varphi,v_x,v_y,v_{\theta},v_{\varphi}).\nonumber
\end{eqnarray*}
In the above, $(\lambda_1,\lambda_2)\in V^{\ast}$ are Lagrange multipliers, where we set $V=\mathbb{R}^{2}$.
\medskip

%
For each point  $x=(q,v,p,\lambda)\in (TQ \oplus T^{\ast}Q) \times V^{\ast}$, we employ the local coordinates
\[
x=(x,y,\theta,\varphi, v_{x}, v_{y},v_{\theta},v_{\varphi},p_{x}, p_{y},p_{\theta},p_{\varphi},\lambda_1,\lambda_2).
\]
Then the Hamilton-Pontryagin principle for $\ext$ yields the equations of motion:
\begin{equation}\label{EqsPes}
\begin{array}{lcl}
&& \dot x=v_x,\quad\dot y=v_y,\quad\dot\theta=v_{\theta},\quad\dot\varphi=v_{\varphi},\\[2mm]
& &p_x=v_x+\lambda_1\,\sin{\theta}, \quad \dot p_x=0, \\[2mm]
& &p_y=v_y-\lambda_1\,\cos{\theta},  \quad \dot p_y=0,\\[2mm]
& &p_{\theta}=I_1\,v_{\theta}, \quad  \dot p_{\theta}=\lambda_1\lp v_x\cos{\theta}+v_y\sin{\theta}\rp+\lambda_2\lp -v_x\sin{\theta}+v_y\cos{\theta}\rp,\\[2mm]
& &p_{\varphi}=I_2\,v_{\varphi}-\lambda_2\,R,\quad \dot p_{\varphi}=0,
\end{array}
\end{equation}
together with the nonholonomic constraints:
\[
v_x\sin{\theta}-v_y\cos{\theta}=0,\,\,\,\,v_x\cos{\theta}+v_y\sin{\theta}-Rv_{\varphi}=0.
\]

Next, we shall see how the vakonomic Lagrange-Dirac system can be constructed by using a Dirac structure $\widehat{D}$ on $T^{\ast}Q\times V^{\ast}$. 
Associated with the vakonomic Lagrangian $\mathfrak{L}(q,v,\lambda)$ on  $TQ\times V^{\ast}$, its differential  is a one-form $\de\ext:TQ\times V^{\ast}\Flder T^{\ast}(TQ\times V^{\ast})$, which may be locally described by
\begin{equation*}
\begin{split}
\de\ext
&=\left(x,\;y,\; \theta,\; \varphi,\;v_x,\; v_y,\; v_{\theta},\; v_{\varphi},\;  \lambda_1,\; \lambda_2,\right.\\
&\qquad \left.\frac{\der\ext}{\der x},\frac{\der\ext}{\der y},\; \frac{\der\ext}{\der \theta},\; \frac{\der\ext}{\der \varphi},\; 
\frac{\der\ext}{\der v_{x}},\; \frac{\der\ext}{\der v_{y}},\; \frac{\der\ext}{\der v_{\theta}},\; \frac{\der\ext}{\der v_{\varphi}},\;\frac{\der\ext}{\der\lambda_1},\frac{\der\ext}{\der\lambda_2}
\right).
\end{split}
\end{equation*}
Then, the vakonomic Dirac differential of $\ext$ is denoted by
\begin{equation*}
\begin{split}
\de_D\ext&=\left(x,\;y,\; \theta,\; \varphi,\; \frac{\der\ext}{\der v_{x}},\; \frac{\der\ext}{\der v_{y}},\; \frac{\der\ext}{\der v_{\theta}},\; \frac{\der\ext}{\der v_{\varphi}},\; \lambda_1,\; \lambda_2,\;\right.\\
&\qquad \left.-\frac{\der\ext}{\der x},-\frac{\der\ext}{\der y},\; -\frac{\der\ext}{\der \theta},\; -\frac{\der\ext}{\der \varphi},\; v_x,\; v_y,\; v_{\theta},\; v_{\varphi},\;\frac{\der\ext}{\der\lambda_1},\frac{\der\ext}{\der\lambda_2}\right).
\end{split}
\end{equation*}
It follows from the condition of the vakonomic Lagrange-Dirac system, namely
\begin{equation*}
((\dot{q}(t),\dot{p}(t),\dot{\lambda}(t)),\mathbf{d}_{D}\mathfrak{L}(q(t),v(t),\lambda(t)))\in \widehat{D}(q(t),p(t),\lambda(t)),
\end{equation*}
that
\[
\lp\begin{array}{cccccccccc}
0&0&0&0&-1&0&0&0&0&0\\[1mm]
0&0&0&0&0&-1&0&0&0&0\\[1mm]
0&0&0&0&0&0&-1&0&0&0\\[1mm]
0&0&0&0&0&0&0&-1&0&0\\[1mm]
1&0&0&0&0&0&0&0&0&0\\[1mm]
0&1&0&0&0&0&0&0&0&0\\[1mm]
0&0&1&0&0&0&0&0&0&0\\[1mm]
0&0&0&1&0&0&0&0&0&0\\[1mm]
0&0&0&0&0&0&0&0&0&0\\[1mm]
0&0&0&0&0&0&0&0&0&0\\[1mm]
\end{array}
\rp
\lp\begin{array}{c}
\dot x\\[1mm] \dot y\\[1mm] \dot\theta \\[1mm] \dot\varphi\\[1mm] \dot p_{x}\\[1mm] \dot p_{y}\\[1mm] \dot p_{\theta}\\[1mm] \dot p_{\varphi}\\[1mm] \dot\lambda_1\\[1mm] \dot\lambda_2
\end{array}
\rp=
\lp\begin{array}{c}
-\frac{\der\ext}{\der x}\\[1mm]-\frac{\der\ext}{\der y}\\[1mm]-\frac{\der\ext}{\der \theta}\\[1mm]-\frac{\der\ext}{\der \varphi}\\[1mm] v_x\\[1mm] v_y\\[1mm] v_{\theta}\\[1mm] v_{\varphi}\\[1mm] \frac{\der\ext}{\der\lambda_1}\\[2mm] \frac{\der\ext}{\der\lambda_2}
\end{array}
\rp,
\]
where
\[
p_{x}=\frac{\der\ext}{\der v_{x}},\quad p_{y}=\frac{\der\ext}{\der v_{y}},\quad p_{\theta}=\frac{\der\ext}{\der v_{\theta}}, \quad p_{\varphi}=\frac{\der\ext}{\der v_{\varphi}}.
\]
Needless to say, the above matrix equation are equivalent with the implicit vakonomic Euler-Lagrange equations given in \eqref{EqsPes}.

\paragraph{The Vakonomic Particle.}
A particle of unit mass evolving in $Q=\R^3$ subject to the nonholonomic constraint $\phi(x,y,z,\dot{x},\dot{y},\dot{z})=\dot z-y\dot x=0$. Using local coordinates $(q,v)=(x,y,z,v_x,v_y,v_z)$ the  Lagrangian is given by $L=\frac{1}{2}\lp v_x^2+v_y^2+v_z^2\rp$, while the vakonomic Lagrangian is
\[
\mathfrak{L}=\frac{1}{2}\lp v_x^2+ v_y^2+v_z^2\rp+\lambda(v_ z-y\,v_x).
\]
For each point $x=(q,v,p,\lambda)\in (TQ\oplus T^*Q)\times V^*$ we employ the local coordinates
\[
x=(x,y,z,v_x,v_y,v_z,p_x,p_y,p_z,\lambda).
\]
The Hamilton-Pontryagin principle for the vakonomic Lagrangian, we obtain the implicit vakonomic Euler-Lagrange equations are given by
\begin{eqnarray*}
v_{x}&=&\dot{x},\qquad \qquad \,\,\,\, v_{y}=\dot{y},\,\,\,\,\,\, \qquad v_{z}=\dot{z},\\
p_x&=&v_x-\lambda\,y,\,\,\,\,\,\,\,\,\,\,\dot p_x=0,\\
p_y&=&v_y,\,\,\,\,\,\,\,\,\,\,\,\,\,\,\,\,\,\,\,\,\,\,\,\,\,\dot p_y=-\lambda\,v_x,\\
p_z&=&v_z+\lambda,\,\,\,\,\,\,\,\,\,\,\,\,\,\,\,\dot p_z=0,
\end{eqnarray*}
together with the nonholonomic constraints
\[
v_z-y\,v_x=0.
\]
Now we construct the implicit vakonomic Lagrangian system using the Dirac structure $\widehat{D}$. Noting 
\[
\de_D\ext=\left(x,\;y,\; z,\; \frac{\der\ext}{\der v_{x}},\; \frac{\der\ext}{\der v_{y}},\; \frac{\der\ext}{\der v_{z}},\; \lambda,-\frac{\der\ext}{\der x},-\frac{\der\ext}{\der y},\; -\frac{\der\ext}{\der z},\; v_x,\; v_y,\; v_z,\;\frac{\der\ext}{\der\lambda}\right),
\]
it follows from the condition 
\begin{equation*}
((\dot{q}(t),\dot{p}(t),\dot{\lambda}(t)),\mathbf{d}_{D}\mathfrak{L}(q(t),v(t),\lambda(t)))\in \widehat{D}(q(t),p(t)\lambda(t))
\end{equation*}
that the implicit vakonomic Euler-Lagrange equations are obtained as
\[
\lp\begin{array}{ccccccc}
0&0&0&-1&0&0&0\\[1mm]
0&0&0&0&-1&0&0\\[1mm]
0&0&0&0&0&-1&0\\[1mm]
1&0&0&0&0&0&0\\[1mm]
0&1&0&0&0&0&0\\[1mm]
0&0&1&0&0&0&0\\[1mm]
0&0&0&0&0&0&0
\end{array}
\rp
\lp\begin{array}{c}
\dot x\\[1mm] \dot y\\[1mm] \dot z \\[1mm]  \dot p_{x}\\[1mm] \dot p_{y}\\[1mm] \dot p_{z} \\[1mm] \dot\lambda
\end{array}
\rp=
\lp\begin{array}{c}
-\frac{\der\ext}{\der x}\\[1mm]-\frac{\der\ext}{\der y}\\[1mm]-\frac{\der\ext}{\der z}\\[1mm] v_x\\[1mm] v_y\\[1mm] v_{z}\\[1mm] \frac{\der\ext}{\der\lambda}
\end{array}
\rp,
\]
where
\[
p_{x}=\frac{\der\ext}{\der v_{x}},\quad p_{y}=\frac{\der\ext}{\der v_{y}},\quad p_{z}=\frac{\der\ext}{\der v_{z}}.
\]

\paragraph{The Vakonomic Skate.}

Consider a plane $\Xi$ with Cartesian coordinates $(x, y)$ of the contact point of the skate with the plane, and slanted at an angle $\alpha$ (which is fixed). Let $\varphi$ be an angle which denotes the orientation of the skate
measured from the $x$ axis. Thus, we shall consider $Q=\R^2\times S^1$ as the configuration manifold of this system. Suppose that the skate is moving under the gravitational force, 
where we denote by $g$ the acceleration due to gravity. Let $m$ and $J$ be the mass and the moment inertia
of the skate about a vertical axis through its contact point respectively. The nonholonomic constraint is given by
\[
\phi(x,y,\varphi, \dot{x},\dot{y},\dot{\varphi})=\sin{\varphi}\,\dot x-\cos{\varphi}\,\dot y=0,
\]
while the mechanical Lagrangian reads
\[
L=\frac{m}{2}\lp \dot x^2+\dot y^2\rp+\frac{J}{2}\dot\varphi^2+m\,g\,x\,\sin{\alpha}.
\]
Using coordinates $x=(q,v,p,\lambda)\in (TQ\oplus T^*Q)\times V^*$, 
\[
x=(x,y,\varphi, v_{x}, v_{y},v_{\varphi},p_{x}, p_{y},p_{\varphi},\lambda).
\]
the vakonomic Lagrangian reads
\[
\mathfrak{L}=\frac{m}{2}(v_x^2+v_y^2)+\frac{J}{2}v_{\varphi}^2+mg\,x\,\sin{\alpha}+\lambda\lp\sin{\varphi}\,v_x-\cos{\varphi}\,v_y\rp.
\]
The Hamilton-Pontryagin principle for the vakonomic Lagrangian induces the implicit vakonomic Euler-Lagrange equations  given by the following set of differential-algebraic equations:
\[
\dot x=v_x, \quad \dot y=v_y, \quad \dot\varphi=v_{\varphi},
\]

\begin{eqnarray*}
p_x&=&mv_x+\lambda\,\sin{\varphi},\,\,\,\,\,\,\,\,\,\,\dot p_x=mg\,\sin{\alpha},\\
p_y&=&mv_y-\lambda\,\cos{\varphi},\,\,\,\,\,\,\,\,\,\dot p_y=0,\\
p_{\varphi}&=&Jv_{\varphi}\,\,\,\,\,\,\,\,\,\,\,\,\,\,\,\,\,\,\,\,\,\,\,\,\,\,\,\,\,\,\,\,\,\,\,\,\,\,\dot p_{\varphi}=\lambda\,(\cos{\varphi}\,v_x+\sin{\varphi}\,v_y),
\end{eqnarray*}
together with the constraints
\[
\sin{\varphi}\,v_x-\cos{\varphi}\,v_y=0.
\]
Now we construct the implicit vakonomic Lagrangian system using the Dirac structure $\widehat{D}$. In this case, one has
\begin{equation*}
\de_D\ext=\left(x,\;y, \;\varphi,\; \frac{\der\ext}{\der v_{x}},\; \frac{\der\ext}{\der v_{y}},\; \frac{\der\ext}{\der v_{\varphi}},\; \lambda,\;-\frac{\der\ext}{\der x},-\frac{\der\ext}{\der y},\; -\frac{\der\ext}{\der \varphi},\; v_x,\; v_y,\; v_{\varphi},\;\frac{\der\ext}{\der\lambda}\right)
\end{equation*}
and it follows from 
\begin{equation*}
((\dot{q}(t),\dot{p}(t),\dot{\lambda}(t)),\mathbf{d}_{D}\mathfrak{L}(q(t),v(t),\lambda(t)))\in \widehat{D}(q(t),p(t),\lambda(t))
\end{equation*}
that the implicit vakonomic Euler-Lagrange equations are obtained as
\[
\lp\begin{array}{ccccccc}
0&0&0&-1&0&0&0\\[1mm]
0&0&0&0&-1&0&0\\[1mm]
0&0&0&0&0&-1&0\\[1mm]
1&0&0&0&0&0&0\\[1mm]
0&1&0&0&0&0&0\\[1mm]
0&0&1&0&0&0&0\\[1mm]
0&0&0&0&0&0&0
\end{array}
\rp
\lp\begin{array}{c}
\dot x\\[1mm] \dot y\\[1mm] \dot\varphi\\[1mm] \dot p_{x}\\[1mm] \dot p_{y}\\[1mm] \dot p_{\varphi}\\[1mm] \dot\lambda
\end{array}
\rp=
\lp\begin{array}{c}
-\frac{\der\ext}{\der x}\\[1mm]-\frac{\der\ext}{\der y}\\[1mm]-\frac{\der\ext}{\der \varphi}\\[1mm] v_x\\[1mm] v_y\\[1mm] v_{\varphi}\\[1mm] \frac{\der\ext}{\der\lambda}
\end{array}
\rp,
\]
where
\[
p_{x}=\frac{\der\ext}{\der v_{x}},\quad p_{y}=\frac{\der\ext}{\der v_{y}},\quad p_{\varphi}=\frac{\der\ext}{\der v_{\varphi}}.
\]

\section{Conclusions and Future Works}
We have explored vakonomic mechanics in the context of  Dirac structures and its associated Lagrange-Dirac systems. First, we have shown that the Lagrangian submanifold theory cannot represent nonholonomic mechanics, but vakonomic mechanics can be properly described on a Lagrangian submanifold. Second, we have shown the Lagrange-Dirac dynamical formalism, especially, employing the symplectomorphisms among the iterated tangent and cotangent bundles $TT^{\ast}Q$, $T^{\ast}TQ$ and $T^{\ast}T^{\ast}Q$. Then, we have defined  a vakonomic Lagrangian on  $TQ\times V^{\ast}$ by incorporating nonholonomic constraints into a given Lagrangian on $TQ$. Moreover, we have introduced  its associated energy on the vakonomic Pontryagin bundle $(TQ \oplus T^{\ast}Q) \times V^{\ast}$. Employing this energy, we have shown  that the Hamilton-Pontryagin principle provides the implicit Euler-Lagrange equations for the vakonomic Lagrangian. We have also shown that one can develop a Dirac structure on $(TQ \oplus T^{\ast}Q) \times V^{\ast}$ and its associated vakonomic Lagrange-Dirac system, which yields the implicit vakonomic Euler-Lagrange equations. Furthermore, we have established another Dirac structure on  $T^{\ast}Q \times V^{\ast}$ by extending the formula given in \cite{YoMa2006a}. To do this, we have introduced the bundle maps $\widehat{\Omega}^{\flat}: T(T^{\ast}Q \times V^{\ast}) \to T^{\ast}(T^{\ast}Q \times V^{\ast})$ and $\tilde{\gamma}_Q: T^{\ast}(T^{\ast}Q \times V^{\ast}) \to T^{\ast}(TQ \times V^{\ast})$ in order to define the Dirac differential for the vakonomic Lagrangian on  $TQ\times V^{\ast}$. Finally, we have illustrated our theory by some examples such as the vakonomic particle, the vakonomic skate and the vertical rolling coin.
\medskip

  We hope that the framework of vakonomic Lagrange-Dirac mechanics proposed in this paper can be explored further.  In particular, the following researches are of our concern for future work:
  
\begin{itemize}
\item {\bf \it Symmetry reduction:}  We are interested in the vakonomic Lagrange-Dirac systems with symmetry (see for instance \cite{MarCordeLe2001}) and is our intention to establish a Dirac reduction theory for this. 
 \item {\bf \it Discrete vakonomic Lagrange-Dirac mechanics:} In parallel with what we have done in the continuous setting of the vakonomic Lagrange-Dirac systems in this paper, its discrete analogue shall be developed.
 \item {\bf \it Applications to optimal control problems:}  Due to the relationship of vakonomic mechanics and optimal control theory, and the applications of the latter to vehicles, space missions design, etc.; it is our aim to explore further vakonomic Lagrange-Dirac systems in this direction.
 \end{itemize}
%
\paragraph{Acknowledgements.}
We are very grateful to David Martin de Diego, Frans Cantrijn, Tom Mestdag and Fran\c{c}os Gay-Balmaz for useful remarks and suggestions. We also greatly appreciate Yoshihiro Shibata for his hearty supporting at Institute of Nonlinear Partial Differential Equations of Waseda University. 
\smallskip

The research of F. J. was supported in its first part by Institute of Nonlinear Partial Differential Equations at
Waseda University and was partially developed during his staying there in 2012 as a visiting Postdoctoral associate. In its second stage, the research of F. J. was supported by the DFG Collaborative Research Center TRR 109, `Discretization in Geometry and Dynamics'. The research of H. Y. is partially supported by JSPS Grant-in-Aid 26400408, JST-CREST,  Waseda University Grant for SR 2012A-602 and IRSES project Geomech-246981.


\end{document}